\documentclass[12pt,reqno]{amsart}
\advance \topmargin by -\headheight
\advance \topmargin by -\headsep
\evensidemargin \oddsidemargin
\usepackage{amsmath}
\usepackage{amsfonts}
\usepackage{stmaryrd}
\usepackage{amssymb}
\usepackage{amsthm}
\usepackage{csquotes}
\usepackage{color}
\usepackage{mathrsfs}
\usepackage{dsfont}
\usepackage{bm}
\usepackage{bbm}
\usepackage{cite}
\usepackage{soul}
\usepackage{enumerate}

\numberwithin{equation}{section}

\newtheorem{theorem}{Theorem}[section]
\newtheorem{fact}[theorem]{Fact}

\newtheorem{lemma}[theorem]{Lemma}

\theoremstyle{definition}

\newtheorem{example}{Example}[section]

\def\RR{\mathbb{R}}
\def\PP{\mathrm{P}}
\def\CC{\mathbb{C}}

\def\GL{\mathrm{GL}}

\def\PP{\mathbb{P}}
\def\e{\mathbbm{1}}

\usepackage{mathtools}

\title[A group-action Szemer\'edi--Trotter and applications to orchard problems]{A group-action Szemer\'edi--Trotter Theorem and applications to orchard problems in all characteristics}

\author{Yifan Jing}
\address{Department of Mathematics, the Ohio State University, Columbus, OH, 43210, USA}
\email{jing.245@osu.edu}

\author{Tingxiang Zou}
\address{Tingxiang Zou, Mathematical Institute, University of Bonn, Endenicher Allee 60, 53115 Bonn, Germany}
\email{tzou@math.uni-bonn.de}

\begin{document}

\begin{abstract} 
We establish a group-action version of the Szemer\'edi--Trotter theorem over any field, extending Bourgain's result for the group $\mathrm{SL}_2(\mathsf{K})$. 
 As an Elekes--Szab{\'o}-type application, we obtain quantitative bounds on the number of collinear triples on reducible cubic surfaces in $\mathbb{P}^3(\mathsf{K})$, where $\mathsf{K} = \mathbb{F}_{q}$ and $\mathsf{K} = \mathbb{C}$, thereby improving a recent result by Bays, Dobrowolski, and the second author.
\end{abstract}
\maketitle

\section{Introduction}

The Szemerédi--Trotter theorem is a fundamental result in additive and extremal combinatorics. The original form of the theorem~\cite{SzT} 
asserts that if $A, B \subseteq \mathbb{R}^2$ are two sets of points in the plane, then the number of solutions to
\begin{equation}\label{eq: Sz-T for line}
 ab + a'b' = 1  \qquad  (a,a',b,b')\in A\times B
\end{equation}
is upper bounded by $O\left(|A|^{\frac{2}{3}} |B|^{\frac{2}{3}} + |A| + |B|\right)$. Over the past decades, generalizations of this theorem have been proven in various settings. In particular, the linear relation in \eqref{eq: Sz-T for line} can be replaced by algebraic curves $\mathcal{C}$ of bounded complexity, and the same upper bound holds if the relation $\mathcal{C}$ is so-called $K_{t,t}$-free for some absolute constant $t\geq2$. This means there are no $t$ distinct pairs $(a, a') \in A$ and $t$ distinct pairs $(b, b') \in B$ such that $\mathcal{C}$ vanishes on all $t^2$ choices of $(a, a', b, b')$. Generalizations to semi-algebraic sets, to $\mathbb{C}$, and to higher-dimensional varieties have also been studied (see~\cite{PS98}, \cite{FPSZ17}, \cite{CGS20}, \cite{Walsh} and so on).

Consider the Cartesian product setting where $A, B \subseteq \mathbb{R}$ are two finite sets of real numbers, with $|A| \geq |B|$. The number of solutions to the equation $ab + a'b' = 1$ with $a, a' \in A$ and $b, b' \in B$ is trivially bounded above by $|A||B|^2$. From a qualitative perspective, the Szemerédi–Trotter theorem suggests that this number should exhibit a power saving of the form $(|B|^2|A|)^{1-\delta}$ compared to the trivial bound, provided that $|B|^{2-c} \gg |A|$ for some fixed positive constant $c$. The size constraints on $|A|$ and $|B|$ arise from the error term in the Szemerédi–Trotter theorem. Specifically, when $|B|^2 \leq |A|$, the error term becomes dominant, leading only to the upper bound $|A|^2$, which is less effective than $|A||B|^2$, the trivial counting bound.

Since the Szemer\'edi--Trotter theorem is now available for algebraic curves over fields of characteristic 0, one can similarly obtain a power-saving counting result of the form $(|B|^2|A|)^{1-\delta}$, even when the line equation~\eqref{eq: Sz-T for line} is replaced by other algebraic curves, provided the ambient field has characteristic $0$. And this holds when the two sets  $A$ and $B$ have comparable sizes,  specifically when $|B|\leq |A|\leq |B|^{2-c}$. 

In light of the sum-product phenomenon, the product theorem~\cite{BGT, PS}, and the theory of approximate groups~\cite{BGTapprox}, one might naturally expect to obtain a similar power-saving counting result across all fields, including those of finite characteristic, without imposing size constraints on $A$ and $B$, provided that the algebraic curve is not ``controlled" by nilpotent groups. Further evidence supporting this heuristic was given by Bourgain~\cite{Bourgain}, who achieved such a result over $\mathbb F_p$, specifically for hyperbolas. He showed that for any fixed $t\in \mathbb{F}_p\setminus\{0\}$, the number of solutions to
\[
(a-b)(a'-b')=t,\qquad  a,a'\in A,\text{ and } b,b'\in B
\]
 is at most $|A|^{1-\delta}|B|^2$ when $|B|\leq |A|\leq |B|^r$ for any fixed constant $r$,  and the power-saving bound $\delta$ depends on $r$, provided that $|A|\ll p^{1-\varepsilon}$ (see \cite{shkredov2018asymptotic} for details). 

Bourgain proved the result by observing that the solutions of the hyperbola can be viewed as a set in $\mathrm{SL}_2$ acting on a set in the projective line. This allowed him to utilize Helfgott's product theorem~\cite{Helfgott} for $\mathrm{SL}_2(\mathbb{F}_p)$ to derive the desired bound. More precisely, he established the following result.

%\textcolor{red}{what is $\ll$?}
\begin{theorem}[Bourgain\cite{Bourgain}]
    For all $\varepsilon>0$ and $r>1$, there is a $\delta>0$ such that the following holds. Let $p$ be a large prime and $A\subseteq\mathbb{F}_p$, $S\subseteq \operatorname{SL}_2(\mathbb{F}_p)$ satisfying the conditions
    \begin{enumerate}
        \item $1\ll |A|\ll p^{1-\varepsilon}$;
        \item $\log |A|<r\log|S|$;
        \item $|S\cap gH|<|S|^{1-\varepsilon}$ for any proper subgroup $H\subset \operatorname{SL}_2(\mathbb{F}_p)$.
    \end{enumerate}
    For $g=(^{a ~b}_{c~d})\in\operatorname{GL}_2(\mathbb{F}_p)$, denote $\Gamma_g\subseteq\mathbb{F}_p^2$ the curve 
    \[cxy-ax+dy-b=0.\]
    Then $|\{(x,y,g)\in A\times A\times S; (x,y)\in\Gamma_g\}|<|A|^{1-\delta}|S|$.
\end{theorem}

Consider the map $\pi:\operatorname{SL}_2(\mathbb{F}_p)\to \operatorname{PSL}_2(\mathbb{F}_p)$ and $\tau: \mathbb{F}^*_p\to\mathbb{P}^1(\mathbb{F}_p)$ sending $x\mapsto [1:x]$ and the natual action of $\operatorname{PSL}_2(\mathbb{F}_p)$ on $\mathbb{P}^1(\mathbb{F}_p)$. Then 
%$(x,y)\in \Gamma_g$ if and only if $\pi(g)(\tau(x))=\tau(y)$.
the above result says that \[|\{(x',y',g')\in \tau(A)\times\tau(A)\times\pi(S):g'(x')=y'\}|\leq |\tau(A)|^{1-\delta}|\pi(S)|.\] As explained before, this is a power-saving result for the groups action $\mathrm{PSL_2}(\mathbb{F}_p)\curvearrowright\mathbb{P}^1(\mathbb{F}_p)$.

The idea of using expansion results in nonabelian groups to obtain power-saving counting bounds for number-theoretic problems has since become a central concept with numerous applications. It is an important ingredient in the landmark work of Bourgain, Gamburd, and Sarnak on the study of the Markoff equation~\cite{BGS}. Shkredov applied it to obtain a sum-product phenomenon for Sidon sets~\cite{Shkredov}, and more recently, this idea was also used in Green's counting results concerning quadratic forms in eight prime variables~\cite{Green}.

Our paper generalises Bourgain's theorem to all subgroups of general linear groups over arbitrary fields. Our tool is based on the main result of \cite{ProductThm}. 

\begin{theorem}\label{thm: main1}
For all $n\in\mathbb{N}^{>0}$, $\varepsilon_1, \varepsilon_2, \varepsilon_3>0$ and $r,k\geq 1$, there are $N\in\mathbb{N}$ and $\delta>0$ such that the following holds.
    Suppose $G$ is a subgroup of $\mathrm{PGL}_{n}(\mathsf{K})$ or $\mathrm{GL}_{n}(\mathsf{K})$ for some field $\mathsf{K}$ and that $G$ acts on some set $T$. 
    Suppose there are finite sets $S\subseteq G$, $X\subseteq T$ with $G=\langle S\rangle$ satisfying: 
\begin{enumerate}
    \item 
    $|X|<|G|^{\frac{1}{k}-\varepsilon_1}$ if $G$ is finite;
    \item
    $1\ll \log|X|<r\log|S|$;
    \item
    $|S\cap gH|<|S|^{1-\varepsilon_2}$ for any proper subgroup $H\leq G$ so that there is subgroup $D\subseteq (S^{-1}S)^{N}$ normal in $H$ and $H/D$ is nilpotent of step at most $n-1$. ($D=\{\mathrm{id}\}$ when $\mathrm{char}(\mathsf{K})=0$.) 
    \item For every $D\subseteq (S^{-1}S)^{N}$ and $D\lhd G$, $|D|\leq |G|/|S|^{\varepsilon_2}$, $G/D$ is not nilpotent of step at most $n-1$. ($D=\{\mathrm{id}\}$ when $\mathrm{char}(\mathsf{K})=0$.)
    \item
    $|\{(x_1,\ldots,x_k):x_i\in X,\mathrm{Stab}_G(x_1,x_2,..,x_k)\neq \mathrm{id}_G\}|\ll |X|^{k-\varepsilon_3}$.
\end{enumerate}
Then $|\{(x,y,g):x,y\in X,g\in S, x=gy\}|<|X|^{1-\delta}|S|$.
\end{theorem}

Note that if $G$ is simple, then condition (4) is automatically satisfied. For the case of $\operatorname{PSL}_2(\mathsf{K})$, we know that $\operatorname{PGL}_2(\mathsf{K})$ acts sharply 3-transitively on $\mathbb{P}^1(\mathsf{K})$ and $[\operatorname{PGL}_2(\mathsf{K}):\operatorname{PSL}_2(\mathsf{K})]$ is $1$ or $2$, hence (5) is satisfied with $k=3$ and any $0<\varepsilon_3<1$. Thus, Bourgain's theorem is a special case of Theorem~\ref{thm: main1}.

Condition (3) and (2) are similar in the sense that they are both saying $S$ should not concentrate on nilpotent-by-finite group, where finite here we mean relatively small, in particular contained in the finite power $S^N$. In an attempt to reduce assumptions, we obtain the following variant of Theorem~~\ref{thm: main1}.

\begin{theorem}\label{thm: main}
For all $n\in\mathbb{N}^{>0}$, $\varepsilon_2, \varepsilon_3>0$ and $r,k\geq 1$, there are $N\in\mathbb{N}$ and $\delta>0$ such that the following holds.
    Suppose $G$ is a subgroup of $\mathrm{PGL}_{n}(\mathsf{K})$ or $\mathrm{GL}_{n}(\mathsf{K})$ for some field $\mathsf{K}$ and that $G$ acts on some set $T$. 
    Suppose there are finite sets $S\subseteq G$, $X\subseteq T$ with $G=\langle S\rangle$ satisfying: 
\begin{enumerate}
    \item
    $1\ll \log|X|<r\log|S|$;
    \item
    $|S\cap gH|<|S|^{1-\varepsilon_2}$ for any subgroup $H\leq G$ so that there is subgroup $D\subseteq (S^{-1}S)^{N}$ normal in $H$ and $H/D$ is nilpotent of step at most $n-1$. ($D=\{\mathrm{id}\}$ when $\mathrm{char}(\mathsf{K})=0$.)
    \item
    $|\{(x_1,\ldots,x_k):x_i\in X, \mathrm{Stab}_G(x_1,x_2,..,x_k)\neq \mathrm{id}_G\}|\ll |X|^{k-\varepsilon_3}$.
\end{enumerate}
Then $|\{(x,y,g):x,y\in X,g\in S, x=gy\}|<|X|^{1-\delta}|S|$.
\end{theorem}

Note that condition (2) in Theorem~\ref{thm: main} is equivalent to conditions (3) and (4) in Theorem~\ref{thm: main1} and the requirement that $S^N\neq G$. However, no requirement for the size of $X$ is imposed in Theorem~\ref{thm: main}. There is no entailment between these two theorems, and they are equivalent when the ambient field has characteristic $0$.\footnote{Since by Jordan's theorem, any finite subgroup of $\operatorname{GL}_n(\mathbb{C})$ has a large index abelian subgroup, where the index only depends on $n$.} 

The proofs of Theorems~\ref{thm: main1} and \ref{thm: main} are effective, featuring a single exponential dependence between $\delta$ and $\varepsilon$. An earlier, ineffective version of a similar result for fields of characteristic 0 was established in~\cite{BDZ} using techniques from model theory. \medskip

\noindent{\bf Applications.} We now turn to an application of Theorems~\ref{thm: main1} and \ref{thm: main} to orchard problems for cubic surfaces across all characteristics. This can be viewed as a step towards understanding the Elekes--Szab\'o theorem, or the Szemer\'edi--Trotter theorem for curves in finite fields, particularly for small sets. For results concerning relatively large sets, see~\cite{TaoExpand}.

The planar orchard problem asks for an arrangement of $n$ points on the plane that maximizes the number of 3-rich lines, where a $k$-rich line is one that passes through exactly $k$ points from the set. This problem has a rich history dating back to 1821. The planar orchard problem was solved by Burr, Grünbaum, and Sloane~\cite{burr1974orchard} for the upper bound, and by Green and Tao~\cite{GT13} for the lower bound.

The orchard problem for points on planar curves in $\mathbb{R}$ was addressed in \cite{EScubic}, building on their earlier groundbreaking result, the Elekes--Szab{\'o} theorem \cite{ES}, which was later generalized to higher arity and dimensions under general position assumptions in \cite{BB21}.
In an effort to remove the general position assumption, Bays, Dobrowolski, and the second author studied the orchard problem on smooth cubic surfaces over $\mathbb{C}$ in \cite{BDZ}, where they developed a tool analogous to Theorem \ref{thm: main} over fields of characteristic 0. Analogously, using Theorem \ref{thm: main}, we obtain a non-trivial upper bound on the exponent for the orchard problem for points on three non-parallel planes in projective three-space over any field.
\begin{theorem}\label{thm: three planeIntro}(See Theorem~\ref{thm: three plane} for a more general version)
    For any $\varepsilon>0$, there is $N\in\mathbb{N}$ and $\delta>0$, such that for all fields $\mathsf{K}$, the following holds for any projective planes $P_1,P_2,P_3\in \mathbb{P}^3(\mathsf{K})$ not intersecting in a same line: \\
    Suppose there are finite sets $X_i\subseteq P_i$ for $i=1,2,3$ with 
    \begin{enumerate}
        \item 
        $1\ll |X_1|=|X_2|=|X_3|$;
        \item 
        $|X_i|\leq p^{1/N}$ if $\operatorname{char}(\mathsf{K})=p>0$;
        \item 
        $|X_i\cap \ell|\leq |X_i|^{1-\varepsilon}$ for any projective line $\ell\subseteq P_i$ for $i=1,2,3$.
    \end{enumerate}
    Then, \[|\{(x_1,x_2,x_3)\in X_1\times X_2\times X_3: x_1,x_2,x_3 \text{ distinct and collinear}\}|\leq |X_1|^{2-\delta}.\]
\end{theorem}

 The above theorem says, if $X_i\subseteq P_i$ is a subset whose size is small compared to the characteristic of the ambient field (when the characteristic is positive) and does not concentrate on lines contained in $P_i$, then $X_1\times X_2\times X_3$ cannot have quadratically-many collinear triples.

 Let us now illustrate that all the assumptions are indeed necessary. Firstly, without the assumption that $X_i$ does not concentrate on lines, there can be $O(|X|^2)$-many triple lines, for example when $X_1,X_2,X_3$ are on three lines contained in a common plane, see \cite{GT-ordinaryLines} and \cite{EScubic}. If $P_1,P_2$ and $P_3$ are intersecting in a common line, then there are examples where $X_i$ does not concentrate on lines, but $X_1\times X_2\times X_3$ contains quadratically-many collinear triples, see \cite[Example 3.6]{BDZ}.

 The assumption on the upper bounds on each of the $|X_i|$ is also necessary, from the following example. 
 \begin{example}
    Suppose $\operatorname{char}(\mathsf{K})=p>0$.  Let $P_1:=\{[0:b:c:d]:b,c,d\in \mathsf{K}\}$, $P_2:=\{[a:0:c:d]: a,c,d\in \mathsf{K}\}$ and $P_3:=\{[a:b:c:c]: a,b,c\in \mathsf{K}\}$. Let $d\in\mathbb{F}_p^*$ be a generator of the cyclic group $(\mathbb{F}_p^*,\cdot,1)$. Choose $N=\lfloor p^\frac{1}{k}\rfloor$ for some fixed $k>1$. Let \begin{align*}
        X_1&:=\{[0:d^i:t:t-1]: i\in[-N,N],t\in \mathbb{F}_p\}\\
        X_2&:=\{[-d^i:0:t:t-1]: i\in[-N,N],t\in \mathbb{F}_p\},\text{ and }\\
        X_3&:=\{[d^i:1:t:t]:i\in[-N,N],t\in\mathbb{F}_p\}.
    \end{align*}
    
     Then $|X_i|\approx p^{1+\frac{1}{k}}$ and \[([0:d^j:z:z-1],[-d^{i+j}:0:z-td^j:z-1-td^j],[d^i:1:t:t])\] is a collinear triple for any $i,j,t,z$. Hence, there are approximately $|X_i|^2$-many collinear triples. One can also verify that for any line $\ell$ contained in $P_i$, either $|X_i\cap \ell|\leq 2N+1$ or $|X_i\cap \ell|\leq p$. Hence, $X_i$ does not concentrate on lines. \hfill$\bowtie$
 \end{example}

 From the above example, to obtained a power-saving counting result, the size of each of the $|X_i|$ must be upper bounded. In the upper bound assumption we have in Theorem~\ref{thm: three planeIntro}, $N$ depends on $\varepsilon$ from the proof. We suspect that $N$ should be a fixed constant.

We also obtain a version of the orchard problem when two points of a collinear triple lie on a fixed smooth quadric surface in $\mathbb{P}^3(K)$ for $\mathsf{K}=\CC$ or a finite field.

\begin{theorem}\label{thm: quadricIntro}(See Theorem~\ref{thm: quadric})
    For any $\varepsilon>0$, there is $N\in\mathbb{N}$ and $\delta>0$, such that for all $\mathsf{K}=\mathbb{F}_{p^n}$ (uniform in $p$ and $n$) or $\mathsf{K}=\CC$, the following holds for any smooth quadratic surface $Q\subseteq \mathbb{P}^3(\mathsf{K})$. Suppose there are finite sets $X\subseteq Q$ and $S\subseteq \mathbb{P}^3(\mathsf{K})\setminus Q$ with 
    \begin{enumerate}
        \item 
        $1\ll |X|=|S|$;
        \item 
        $|X|\leq p^{1/N}$ if $\operatorname{char}(\mathsf{K})=p>0$;
        \item 
        $\max\{|X\cap \ell|, |S \cap\ell|\}\leq |X|^{1-\varepsilon}$ for any projective line $\ell\subseteq  \mathbb{P}^3(\mathsf{K})$.
    \end{enumerate}
    Then, \[|\{(x_1,x_2,x_3)\in X^2\times S: x_1,x_2,x_3 \text{ distinct and collinear}\}|\leq |X|^{2-\delta}.\]
\end{theorem}

As our proof is effective, the power-saving constant $\delta$ we obtained from both Theorem~\ref{thm: three planeIntro} and Theorem~\ref{thm: quadricIntro} grows exponentially on $\varepsilon$. We remark that in~\cite{BDZ}, the authors posed the question of whether $\delta$ can be chosen to depend linearly on $\varepsilon$, and as of yet, the answer to this question remains unknown.\medskip

\noindent{\bf Organization of the paper.} In Section 2, we will review Bourgain--Gamburd's $L^2$-flattening lemma and present a modification of it. In Section 3, we prove Theorems~\ref{thm: main1} and \ref{thm: main}. We then move to the application of these results to the orchard problem for cubic surfaces. In Section 4, we address the three planes case and prove Theorem~\ref{thm: three planeIntro}. Finally, in Section 5, we complete the proof of Theorem~\ref{thm: three planeIntro}.
\medskip

\noindent{\bf Acknowledgments.} The authors would like to thank Ben Green and Peter Sarnak for insightful discussions regarding generalizations of the main theorem in \cite{Bourgain}. We also thank Akshat Mudgal for drawing our attention to \cite{Murphyline} and for related discussions, and Ilya Shkredov for his careful reading of the manuscript and many valuable comments. We would also like to thank Weikun He for helpful discussions around finite subgroups of general linear groups, and Chuzhan Jin for drawing the figure. Additionally, we are grateful to Gabriel Conant, Caroline Terry, and Julia Wolf for organizing the \emph{Combinatorics Meets Model Theory} workshop at Cambridge in the summer of 2022, which made this collaboration possible.

\section{A modification of Bourgain--Gamburd's $L^2$-flattening}

\subsection{Notation and Preliminaries}
Given a group $G$ and a subset $A\subseteq G$, we denote $A^n$ the $n$-fold product of $A$, namely $A^n:=\{g_1\cdots g_n:g_i\in A\}$. For the Cartesian product $A\times\cdots \times A$. To avoid confusion, we will use the unconventional $A^{\times k}$ in section \ref{sec: group action}.

Given a group $G$ and a finitely supported function $f:G\to\RR$ we define the $L^p$-norm of $f$ to be
\[
\|f\|_{L^p(G)}= \Big(\sum_g |f(g)|^p\Big)^{\frac{1}{p}}.
\]
As usual, given two finitely supported functions $f,h:G\to\RR$, the convolution is defined by
\[
f*g(x)=\sum_y f(y)g(y^{-1}x).  
\]
Given a set $A\subseteq G$ we use $\e_A$ to denote the indicator function of $A$. That is $\e_A(x)=1$ when $x\in A$, and otherwise $\e_A(x)=0$. Given a finite set $A$, we write $\mu_A=\frac{1}{|A|}\e_A$, which is a probability measure supported on $A$.

We will use the following version of the Balog--Szemer\'edi--Gowers theorem for measures in nonabelian groups. 
We include a proof in the appendix for the sake of completeness.
\begin{theorem}\label{thm: BSG}
Let $G$ be a group, and $K\geq 1$ be an integer. Suppose $\mu:G\to[0,1]$ is a finitely supported symmetric probability measure and $\|\mu*\mu\|_{L^2(G)}\geq K^{-1}\|\mu\|_{L^2(G)}$. 
Then there is an $O(K^{O(1)})$-approximate subgroup $H$ of $G$ such that
\begin{enumerate}
    \item  $|H|\ll K^{O(1)}/\|\mu\|_{L^2(G)}^2$;
    \item  $H\subseteq \mathrm{supp}(\mu)^6$;
    \item  $\mu(xH)\gg K^{-O(1)}$ for some $x\in G$.
\end{enumerate}
\end{theorem}

We will also use the following version of the product theorem which is proven recently by Eberhard--Murphy--Pyber--Szab\'o~\cite{ProductThm}. The statement for field of characteristic $0$ was proven by~\cite{BGT,PS}. 
\begin{fact}\label{fact: product-theorem}
Let $\mathsf{K}$ be a field, $G\leq \GL_n(\mathsf{K})$, and $A\subseteq G$ a finite $K$-approximate group. Then there are $D\lhd \Gamma\lhd \langle A\rangle$ such that 
\begin{enumerate}
    \item $A$ can be covered by $K^{O_n(1)}$ translates of $\Gamma$;
    \item $D$ is contained in $A^{O_n(1)}$;
    \item $\Gamma/D$ is a nilpotent group of step at most $n-1$. 
\end{enumerate}
Moreover, when $\mathrm{char}(\mathsf{K})=0$, $D$ is trivial. 
\end{fact}

Let us finally remark that as $\mathrm{PGL}_n(\mathsf{K})$ is a subgroup of $\mathrm{GL}_{n^2}(\mathsf{K})$, Fact~\ref{fact: product-theorem} applies naturally for finite approximate groups of $\mathrm{PGL}_n(\mathsf{K})$.

\subsection{$L^2$-flattening}
We are going to prove the following theorem.
\begin{theorem}\label{thm: l2 flattening} 
For any $n\in\mathbb{N}^{>0}$, there is $c>0$ and $N=N(n)$ such that the following holds for any $K>0$. 
Suppose $G$ is a subgroup of $\mathrm{GL}_n(\mathsf{K})$ or $\mathrm{PGL}_n(\mathsf{K})$ for some field $\mathsf{K}$. Suppose $\mu$ is a symmetric finitely supported probability measure on $G$. 
Let $\mathcal{D}:=\{D\leq G: D\subseteq \mathrm{supp}(\mu)^N\}$ when $\mathrm{char}(\mathsf{K})=p>0$ and let $\mathcal{D}:=\{\{\mathrm{id}\}\}$ when $\mathrm{char}(\mathsf{K})=0$.

Suppose (*)
%(Non-concentration on nilpotent by finite groups) 
for every subgroup $\Gamma\leq G$, if there is $D\lhd \Gamma$ so that $\Gamma/D$ is nilpotent of step at most $n-1$, and $D\in\mathcal{D}$, then
    $\mu(x\Gamma)\leq K^{-1}$ for all $x\in G$. 
 
Or suppose the following holds:
\begin{enumerate}
    \item (Non-concentration) For every proper subgroup $\Gamma\leq G$, if there is $D\lhd \Gamma$ so that $\Gamma/D$ is nilpotent of step at most $n-1$, and $D\in\mathcal{D}$, then
    $\mu(x\Gamma)\leq K^{-1}$ for all $x\in G$; 
    \item (Non-nilpotent by finite) For every $D\lhd G$ with $|D|< |G|/K$ and $D\in\mathcal{D}$, $G/D$ is not a nilpotent group of step at most $n-1$; 
    \item (Non-uniform) If $G$ is finite, then $\|\mu\|_{L^2(G)}>K|G|^{-1/2}$. 
\end{enumerate}
Then we have
\[
\|\mu*\mu\|_{L^2(G)}\ll K^{-c}\|\mu\|_{L^2(G)}.
\]
\end{theorem}
\begin{proof}
Suppose $\|\mu*\mu\|_{L^2(G)}\gg K^{-c}\|\mu\|_{L^2(G)}$ for some $c>0$ to be determined later. By the Balog--Szemer\'edi--Gowers theorem (Theorem~\ref{thm: BSG}) there are constants $C_1,C_2,C_3>0$ such that 
\begin{itemize}
    \item there is a $K^{cC_1}$-approximate group $H$;
    \item there is $x\in G$ such that $\mu(xH)\gg K^{-cC_2}$;
    \item $|H|\ll K^{cC_3}/\|\mu\|_{L^2(G)}^2$;
    \item $H\subseteq \mathrm{supp}(\mu)^6$. 
\end{itemize}

Using the product theorem (Fact~\ref{fact: product-theorem}), we obtained groups $D,\Gamma$ and constants $C_4,C_5>0$ only depending on $n$ such that 
\begin{itemize}

    \item $D\lhd \Gamma \lhd \langle H\rangle$; 
    \item $H$ can be covered by $K^{cC_4}$ translates of $\Gamma$;
    \item $D$ is contained in $H^{C_5}$; 
    \item $\Gamma/D$ is a nilpotent group of step at most $n-1$;
    \item $D=\{\mathrm{id}\}$ when $\mathrm{char}(\mathsf{K})=0$.
\end{itemize}

Note that $D\subseteq H^{C_5}\subseteq \mathrm{supp}(\mu)^{6C_5}$. Let $N:=6C_5$. Then $\mu(xH)\leq\sum_{i\leq N}\mu(xg_i\Gamma)$ where $g_i\in G$ and $N\leq K^{cC_4}$. Then by the pigeon-hole principle there is $i\leq N$ such that
\[
K^{cC_4}\mu(xg_i\Gamma)\geq \mu(xH) \gg K^{-cC_2}.
\]
Thus when $c<1/(C_2+C_4)$, we conclude that $\mu(xg_i\Gamma)\gg K^{-1}$, which contradicts (*) and the non-concentration assumption (1) if $\Gamma\neq G$.

Therefore, we only need to consider the case when $\Gamma=G$ with assumptions (2) and (3). 

If $G$ is infinite, then $|D|< |G|/K$ since $D$ is finite. We get $G/D$ is not nilpotent of step at most $n-1$, a contradiction. 

If $G$ is finite, then $\|\mu\|_{L^2(G)}>K|G|^{-1/2}$ by assumption (3). Hence, we have
\[
|H|\ll \frac{K^{cC_3}}{K^2}|G|.
\]
By assumption (2), we must have $|D|\geq |G|/K$. Thus, 
\[
|G|\leq K|D|\leq K|H^{C_5}|\leq K^{1+cC_1C_5}|H|\ll K^{c(C_3+C_1C_5)-1}|G|,
\]
and this is impossible if we choose $c<1/(C_3+C_1C_5)$. 

In conclusion, we have proved the theorem with $c<\min\{1/(C_3+C_1C_5),1/(C_2+C_4)\}$ and $N=6C_5$.
\end{proof}

\section{Proofs of the group-action Szemer\'edi--Trotter theorem}\label{sec: group action}
In this section, we prove Theorem~\ref{thm: main1} and Theorem~\ref{thm: main} together. Let's first work with the common assumptions of these two theorems. Let $G$ be a subgroup of $\mathrm{PGL}_{n}(\mathsf{K})$ or $\mathrm{GL}_{n}(\mathsf{K})$ for some field $\mathsf{K}$ and that $G$ acts on some set $T$. Let $S\subseteq G$ and $X\subseteq T$ be finite subsets with $G=\langle S\rangle$.
Suppose 
\begin{equation}\label{ass: r}
    1\ll \log|X|<r\log|S|,
\end{equation}
 and
\begin{equation}\label{ass: Stab}
  |\{(x_1,\ldots,x_k):\mathrm{Stab}_G(x_1,x_2,..,x_k)\neq \mathrm{id}_G\}|\ll |X|^{k-\varepsilon_3}.
\end{equation}

Let $\delta>0$, and assume that we have \[|\{(x,y,g):x,y\in X,g\in S, x=gy\}|\geq|X|^{1-\delta}|S|,\] namely
\[\sum_{g\in G}\mu_S(g)|X\cap gX|\geq |X|^{1-\delta},\]
 where $\mu_S$ is the counting measure supported on $S$. We want to arrive at a contradiction for some carefully chosen $\delta$ with the further assumptions of Theorem~\ref{thm: main1} and Theorem~\ref{thm: main}.
 
Note that $X^{\times k}\cap gX^{\times k}=(X\cap gX)^{\times k}$ (recall that $X^{\times k}$ is the Cartesian product of $X$).
Hence by Jensen's inequality we have 
\begin{equation}\label{eq: assumption}
\sum_{g\in G}\mu_S(g)|X^{\times k}\cap gX^{\times k}|=\sum_{g\in G}\mu_S(g)|X\cap gX|^k\geq \left(\sum_{g\in G}\mu_S(g)|X\cap gX|\right)^k \geq |X^{\times k}|^{1-\delta}. 
\end{equation}
Let $
\widetilde{X}=\{x\in X^{\times k}\mid \mathrm{Stab}_G(x)=\mathrm{id}_G\}
$ and $
\widetilde{T}=\{x\in T^{\times k}\mid \mathrm{Stab}_G(x)=\mathrm{id}_G\}. 
$
%Hence we have $\widetilde{X}\subseteq\widetilde{V}$. 

Note that $g\widetilde{T}=\widetilde{T}$ and $g(T^k\setminus \widetilde{T})=T^k\setminus \widetilde{T}$ for all $g\in G$. Hence, \[X^{\times k}\cap gX^{\times k}=(\widetilde{X}\cap g\widetilde{X})\cup((X^{\times k}\setminus \widetilde{X})\cap g(X^{\times k}\setminus \widetilde{X})).\]
Thus by \eqref{eq: assumption} and (\ref{ass: Stab}) we get
\begin{equation}\label{eq: assumption for X tilde}
  \sum_{g\in G}\mu_S(g)|\widetilde{X}\cap g\widetilde{X}|\gg |X^{\times k}|^{1-\delta}-   |X|^{k-\varepsilon_3}\gg |\widetilde{X}|^{1-\delta}
\end{equation}
as long as $
\delta<\varepsilon_3/k. $

On the other hand, we note that
\begin{align*}
  \sum_{g\in G}\mu_S(g)|\widetilde{X}\cap g\widetilde{X}|=\sum_{g\in G}\mu_S(g)\sum_{x\in \widetilde{T}}\e_{\widetilde{X}}(x)\e_{\widetilde{X}}(g^{-1}x)=\sum_{x\in \widetilde{T}}\e_{\widetilde{X}}(x)\left(\sum_{g\in G}\mu_S(g)\e_{\widetilde{X}}(g^{-1}x)\right). 
\end{align*}
Therefore by the Cauchy--Schwarz inequality we have
\begin{align*}
  \left(\sum_{g\in G}\mu_S(g)|\widetilde{X}\cap g\widetilde{X}|\right)^2\leq |\widetilde{X}|\sum_{g_1,g_2\in G}\sum_{x\in \widetilde{T}}\mu_S(g_1)\mu_S(g_2)\e_{\widetilde{X}}(g_1^{-1}x)\e_{\widetilde{X}}(g_2^{-1}x). 
\end{align*}
By a change of variable, the right hand side of the above inequality can be rewritten as
\begin{align*}
    &|\widetilde{X}|\sum_{g_1,g_2\in G}\mu_S(g_1)\mu_S(g_2)|g_1\widetilde{X}\cap g_2\widetilde{X}|=|\widetilde{X}|\sum_{g_1\in G}\mu_S(g_1)\sum_{g_2\in G}\mu_S(g_2)|\widetilde{X}\cap g_1^{-1}g_2\widetilde{X}|\\
    &=|\widetilde{X}|\sum_{g_1\in G}\mu_S(g_1)\sum_{g}\mu_S(g_1g)|\widetilde{X}\cap g\widetilde{X}|
    =|\widetilde{X}|\sum_{g\in G}\sum_{g_1\in G}\mu_S(g_1)\mu_S(g_1g)|\widetilde{X}\cap g\widetilde{X}|
    \\
    &=
    |\widetilde{X}|\sum_g \widetilde{\mu_S}*\mu_S(g) |\widetilde{X}\cap g\widetilde{X}|,
\end{align*}

where $\widetilde{\mu_S}(g)=\mu_S(g^{-1})$.
Together with \eqref{eq: assumption for X tilde} we have
\[
\sum_{g\in G}\widetilde{\mu_S}*\mu_S(g) |\widetilde{X}\cap g\widetilde{X}|\geq |\widetilde{X}|^{1-2\delta}. 
\]
For even $m$, we write $\mu_S^{(m)}$ for the $m$-fold symmetric convolution $\widetilde{\mu_S}*\mu_S*\mu_S^{(m-2)}$. Note that  $\widetilde{\mu_S}*\mu_S$ is symmetric, so is $\mu_S^{(m)}$. And $\mu_S^{(m)}$ is supported on $(S^{-1}S)^{m/2}$.
%namely $\mu_S^{(2^m)}=\widetilde{\mu_S^{(2^{m-1})}}*\mu_S^{(2^{m-1})}$. 
By induction, we conclude that
\begin{equation}\label{eq: n-fold convolution}
\sum_{g\in G}\mu_S^{(2^m)}(g) |\widetilde{X}\cap g\widetilde{X}|\geq |\widetilde{X}|^{1-2^m\delta}. 
\end{equation}

Now we consider the following set $\Omega_t$ in $G$ such that each element in $\Omega_t$ makes $\widetilde{X}$ almost invariant. More precisely, for some $t<1$ to be determined later, we define
\[
\Omega_t=\Big\{g\in G : |g\widetilde{X}\cap \widetilde{X}|>\frac{1}{2}|\widetilde{X}|^{1-t}\Big\}. 
\]
We will see that $\Omega_t$ cannot be too large. Indeed, consider 
\[
\sum_{g\in\Omega_{t}}\sum_{x\in \widetilde{X}}\e_{\widetilde{X}}(g^{-1}x). 
\]
By the definition of $\Omega_{t}$ the above quantity is lower bounded by $\frac{1}{2}|\Omega_{t}||\widetilde{X}|^{1-t}$. 
On the other hand by the definition of $\widetilde{X}$, for every $x\in \widetilde{X}$ and $y\in \widetilde{X}$ there is a unique $g\in G$ such that $gx=y$. Therefore,
\[
\frac{1}{2}|\Omega_{t}||\widetilde{X}|^{1-t}\leq \sum_{g\in\Omega_{t}}\sum_{x\in \widetilde{X}}\e_{\widetilde{X}}(g^{-1}x)\leq \sum_{x\in \widetilde{X}}\sum_{g\in G} \e_{\widetilde{X}}(g^{-1}x)=|\widetilde{X}|^2
\]
and hence 
\begin{equation}\label{eq: upper bound Omega}
    |\Omega_{t}|\leq 2|\widetilde{X}|^{1+t}.
\end{equation}

We will show that, because of \eqref{eq: n-fold convolution}, $\Omega_m:=\Omega_t\cap\mathrm{supp}(\mu_S^{(2^m)})$ is large for a carefully chosen $t$ and $m$. Note that for every $m\geq 1$ we have
\[
\sum_{g\in \Omega_t}\mu_S^{(2^m)}(g) |\widetilde{X}\cap g\widetilde{X}|=\sum_{g\in G}\mu_S^{(2^m)}(g) |\widetilde{X}\cap g\widetilde{X}| - \sum_{g\notin \Omega_t}\mu_S^{(2^m)}(g) |\widetilde{X}\cap g\widetilde{X}|.
\]
By the definition of $\Omega_t$,
%the second term of the right hand side of the above equality is at most
\[
\sum_{g\notin \Omega_t}\mu_S^{(2^m)}(g) |\widetilde{X}\cap g\widetilde{X}|\leq\frac{1}{2}|\widetilde{X}|^{1-t}\sum_{g\in G} \mu_S^{(2^m)}(g)=\frac{1}{2}|\widetilde{X}|^{1-t}. 
\]
Thus we have that
\begin{equation}\label{eq: sum over Omega}
    \sum_{g\in \Omega_t}\mu_S^{(2^m)}(g) |\widetilde{X}\cap g\widetilde{X}| \geq \frac{1}{2}|\widetilde{X}|^{1-t}
\end{equation}
whenever
$2^m\delta\leq t.$

Let $\mathcal{D}:=\{D\subseteq (S^{-1}S)^{N}, D\leq G\}$ if $\mathrm{char}(K)>0$ and $\mathcal{D}=\{\{\mathrm{id}\}\}$ if $\mathrm{char}(\mathsf{K})=0$, and let
$\mathcal{H}:=\{H\leq G: \text{ there is } D\in\mathcal{D}\text{ such that } H/D \text{ nilpotent of step at most } n-1;\}$

Now we have two set of assumptions: the first set coming from
Theorem~\ref{thm: main}:
\begin{equation}\label{ass: theorem 1.3}
   |S\cap gH|<|S|^{1-\varepsilon_2} \text{ for any }H\in\mathcal{H}. 
\end{equation}

And the second set of assumptions coming from those of Theorem~\ref{thm: main1}:
 \begin{align}
\label{ass: theorem 1.2 1}& 1\ll|X|<|G|^{\frac{1}{k}-\varepsilon_1}\text{ if } G \text{ is finite};\\ 
\label{ass: theorem 1.2 2}&|S\cap gH|<|S|^{1-\varepsilon_2} \text{ for any }H\in\mathcal{H}\text{ and }H\neq G; \\
\label{ass: theorem 1.2 3}&G/D \text{ is not nilpotent of step}\leq n-1, \text{ for any }D\in\mathcal{D}, D\trianglelefteq G\text{ and } |D|\leq |G|/|S|^{\varepsilon_2}.     
 \end{align}

We need to show that under assumption (\ref{ass: theorem 1.3}) or assumptions (\ref{ass: theorem 1.2 1}, \ref{ass: theorem 1.2 2}, \ref{ass: theorem 1.2 3}) we get contradictions for a carefully chosen $\delta>0$.

We choose some $\varepsilon_4\leq \varepsilon_2$. 
Note that for any subgroup $H\leq G$, if we have $|S\cap gH|<|S|^{1-\varepsilon_2}\leq |S|^{1-\varepsilon_4}$, then this can be rewritten as
$
\mu_S(gH)<|S|^{-\varepsilon_4}. 
$
Hence for $m\geq 1$, we have 
\begin{align}\label{ineq-cond3}
    \mu^{(2^m)}_S(gH)&=\mu^{(2^m-1)}_S*\mu_S(gH)
    =\sum_{x\in gH}\sum_{y\in G}\mu^{(2^m-1)}_S(y)\mu_S(y^{-1} x)\nonumber \\
    &=\sum_{y\in G}\mu^{(2^m-1)}_S(y)\sum_{x\in gH} \mu_S(y^{-1} x) =\sum_{y\in G}\mu^{(2^m-1)}_S(y)\mu_S(y^{-1}gH)<|S|^{-\varepsilon_4}.
\end{align}

We apply Theorem~\ref{thm: l2 flattening} with $\mu_S^{(2^{m})}$ and $|S|^{\varepsilon_4}$. Let $N_0=N(n)$ given by Theorem~\ref{thm: l2 flattening}.  By (\ref{ineq-cond3}),
under the assumption (\ref{ass: theorem 1.3}), we must have $L^2$-flattening for $\mu_S^{(2^{m})}$ for any $m\geq 1$ with $2^mN_0\leq N$ (since $\mathrm{supp}(\mu_S^{(2^{m+1})})\subseteq (S^{-1}S)^{2^{m}}$), namely
    \begin{equation}
    \Vert \mu_S^{(2^{m+1})}\Vert_{L^2(G)} \ll |S|^{-c\varepsilon_4}\Vert\mu_S^{(2^{m})}\Vert_{L^2(G)} \ll |S|^{-cm\varepsilon_4} \Vert \widetilde{\mu}_S*\mu_S\Vert_{L^2(G)}.
\end{equation}
where $c=c(n)$ is the constant obtained from Theorem~\ref{thm: l2 flattening} only depending on $n$.\footnote{Note that we use induction of step $m$, therefore, $\Vert \mu_S^{(2^{m})}\Vert_{L^2(G)} \leq O(1)^m|S|^{-cm\varepsilon_4} \Vert \widetilde{\mu}_S*\mu_S\Vert_{L^2(G)}$.}

Note that by H\"older's inequality, and the fact that $\widetilde{\mu}_S*\mu_S$ is a probability measure,
 \[
 \Vert \widetilde{\mu}_S*\mu_S\Vert_{L^2(G)}\leq \max_x (\widetilde{\mu}_S*\mu_S(x))^{\frac12} \Vert \widetilde{\mu}_S*\mu_S\Vert_{L^1(G)}^{\frac12} = |S|^{-1}\max_x |S\cap Sx|^{\frac12}\leq |S|^{-\frac12}.
 \]
 Hence, 
 \begin{equation}\label{eq: l2 of mu_S}
   \Vert \mu_S^{(2^{m+1})}\Vert_{L^2(G)}\ll |S|^{-(cm\varepsilon_4+1/2)}.   
 \end{equation}

Also note that using Cauchy--Schwarz we obtain the following bound
\begin{equation}\label{LinfL2}
    \| \mu_S^{(2^{m+1})}\|_{L^\infty(G)}=\sup_x \sum_y \mu_S^{(2^{m})}(y)\mu_S^{(2^{m})}(y^{-1}x)\leq \|\mu_S^{(2^{m})}\|_{L^2(G)}^2.
\end{equation}

%Let $\ell=O_{r,k,n,\varepsilon_1,\varepsilon_2}(1)$ to be determined later.  We first assume that for $j\leq \ell$, $\mu_S^{(2^j)}$ exhibits flattening. That is, \textcolor{red}{where is 2?}
Thus, by \eqref{eq: l2 of mu_S} and \eqref{ass: r},
\[
\|\mu_S^{(2^{m+1})}\|_{L^\infty(G)}\leq\|\mu_S^{(2^{m})}\|_{L^2(G)}^2\ll |S|^{-(cm\varepsilon_4+1/2)}\ll  |X|^{-\frac{cm\varepsilon_4+1/2}{r}}.
\]
Now by \eqref{eq: sum over Omega}, when $2^{m+1}\delta< t$,
\begin{equation*}
|\widetilde{X}||\Omega_{t}| \|\mu_S^{(2^{m+1})}\|_{L^\infty(G)}\geq \sum_{g\in \Omega_t}\mu_S^{(2^{m+1})}(g) |\widetilde{X}\cap g\widetilde{X}|\gg |\widetilde{X}|^{1-t}.
\end{equation*}
and together with the trivial bound that $|X|^k\geq|\widetilde{X}|$, 
\begin{equation}\label{eq: lower bound on Omega}
|\Omega_{t}|\gg  |X|^{\frac{cm\varepsilon_4+1/2}{r}}|\widetilde{X}|^{-t}\gg |\widetilde{X}|^{2-t+\frac{1}{2rk}},
\end{equation}
provided that $m\geq \frac{2kr}{c\varepsilon_4}$ and $|\widetilde{X}|\gg1$. Choose
\begin{equation}\label{eq: restriction on l}
m= \left\lceil\frac{2kr}{c\varepsilon_4}\right\rceil,
\end{equation}
then the lower bound \eqref{eq: lower bound on Omega} and the upper bound on $\Omega_t$ \eqref{eq: upper bound Omega} are contradictory whenever we choose $t<1/2$. 

Here we conclude the proof of Theorem~\ref{thm: main}, with $\varepsilon_4:=\varepsilon_2$ and \[\delta<\min\{\frac{\varepsilon_3}{k}, 2^{-\lceil\frac{2kr}{c\varepsilon_2}\rceil-2}\}\text{ and } N:=2^{\lceil\frac{2kr}{c\varepsilon_2}\rceil} N_0.\]

Now we proceed to the proof of Theorem~\ref{thm: main1}, namely with assumptions (\ref{ass: theorem 1.2 1}, \ref{ass: theorem 1.2 2}, \ref{ass: theorem 1.2 3}). Let $\varepsilon_4:=\min\{\frac{k\varepsilon_1}{5},\varepsilon_2, (1-k\varepsilon_1)/2\}$. We apply Theorem~\ref{thm: l2 flattening} with $\mu_S^{(2^{m})}$ and $|S|^{\varepsilon_4}$ again. By (\ref{ass: theorem 1.2 2}, \ref{ass: theorem 1.2 3}), for every $m\geq1$ with $2^mN_0\leq N$, we have one of the following,
\begin{enumerate}
    \item $L^2$-flattening for $\mu_S^{(2^{m})}$, i.e.,
    $\Vert \mu_S^{(2^{m+1})}\Vert_{L^2(G)} \ll |S|^{-c\varepsilon_4}\Vert\mu_S^{(2^{m})}\Vert_{L^2(G)}$;
    \item $\mu_S^{(2^{m})}$ is sufficiently uniform, that is
    \begin{equation}
\Vert \mu_S^{(2^{m})}\Vert_{L^2(G)}\ll |S|^{\varepsilon_4}|G|^{-1/2},
\end{equation}
\end{enumerate} 
Suppose $\mu_S^{(2^{i})}$ has $L^2$-flattening for all $i\leq m$, then by the same argument as before, we get a contradiction, when $m\geq \lceil 2kr/c\varepsilon_4\rceil$, $\delta<\min\{\varepsilon_3/k, 2^{-m-2}\}$ and $N\geq 2^mN_0$.

Otherwise, there is $i\leq m\leq 2kr/c\varepsilon_4$ such that  $\Vert \mu_S^{(2^{i})}\Vert_{L^2(G)}\ll |S|^{\varepsilon_4}|G|^{-1/2}$. Therefore, by (\ref{LinfL2}) and Young's inequality, \[\|\mu_S^{(2^{m+1})}\|_{L^\infty(G)}\leq \|\mu_S^{(2^{m})}\|_{L^2(G)}^2\leq \|\mu_S^{(2^{i})}\|_{L^2(G)}^2  \ll |S|^{2\varepsilon_4}/|G|.\]
Similarly by \eqref{eq: sum over Omega} we get $|\Omega_t|\gg |\widetilde{X}|^{-t}|G||S|^{-2\varepsilon_4}$ when $2^{m+1}\delta<t$ and by the upper bound on $\Omega_t$ \eqref{eq: upper bound Omega} and assumption \eqref{ass: theorem 1.2 1}, \[|G||S|^{-2\varepsilon_4}\ll |G|^{(1-k\varepsilon_1)(1+2t)},\] which is a contradiction whenever
$
t< \frac{k\varepsilon_1-2\varepsilon_4}{2-2k\varepsilon_1}. 
$
As $4\varepsilon_4< k\varepsilon_1$, we choose
\[
t=\frac{\varepsilon_4}{1-k\varepsilon_1}
.\]
Note that by our choice of $\varepsilon_4$, we have $2t<1$. Hence we choose $m=\lceil \frac{2kr}{c\varepsilon_4}\rceil$ and finished the proof of Theorem~\ref{thm: main1} with 
\[\delta = \min\Big\{\frac{\varepsilon_4}{(1-k\varepsilon_1)2^{\lceil\frac{2kr}{c\varepsilon_4}\rceil+1}},\frac{\varepsilon_3}{k}\Big\} \text{ and }N=2^{\lceil \frac{2kr}{c\varepsilon_4}\rceil}N_0.\]

\section{Non-collinear points from three planes}

We will follow the strategy used in \cite{BDZ} and identify collinear triples as group actions.
Throughout the section, we work in the projective three-space $\PP^3(\mathsf{K})$ for some field $\mathsf{K}$.
Given two points $a,b$ we use $\ell_{a,b}$ to denote the line passing through $a$ and $b$. Let $P_1,P_2$ be two distinct projective planes. Given any point $x\not\in P_1\cup P_2$, it induces a bijection $\eta_x:P_1\to P_2$ by sending $a\in P_1$ to the unique point $b:=\ell_{a,x}\cap P_2$. Note that $a=b$ when $a\in P_1\cap P_2$. Given $x,y\not\in P_1\cup P_2$, we have an action $\eta_{y}^{-1}\circ\eta_{x}:P_1\to P_1$. 

We have the following description of this action for all fields. 
\begin{lemma}
    The group generated by $\gamma_{x,y}:=\eta_{y}^{-1}\circ\eta_{x}$ for $x,y\not\in P_1\cup P_2$ can be naturally identified as $G_a(\mathsf{K})^2\rtimes G_m(\mathsf{K})$. 
\end{lemma}
\begin{proof}
    Apply a projective linear transformation, we may assume $P_1=\{[0:b:c:d]: b,c,d\in \mathsf{K}\}$ and $P_2=\{[a:0:c:d]: a,c,d\in \mathsf{K}\}$. Suppose $x=[a:1:c:d]$ and $y=[1:b':c':d']$, let $p=[0:\xi_1:\xi_2:\xi_3]$, then 
    \[
    (\eta_{y}^{-1}\circ \eta_{x})(p)=[0:b'a\xi_1:\xi_2+(c'a-c)\xi_1:\xi_3+(d'a-d)\xi_1].
    \]
    Namely, $\gamma_{x,y}$ corresponds to $(c'a-c,d'a-d,b'a)\in G_a(K)^2\rtimes G_m(K)$. And its action on $P_1$ is given by $(a,b,c)\star[0:\xi_1:\xi_2:\xi_3]=[0:c\xi_1:\xi_2+a\xi_1:\xi_3+b\xi_1]$.
\end{proof}

The next lemma relates the collinearity of three points to solutions to group actions. 

\begin{lemma}\label{lem: S}
Let $P_1,P_2$ be two planes, $X_i\subseteq P_i$ for $i=1,2$ and $X_3\subseteq \mathbb{P}^3(\mathsf{K})\setminus (P_1\cup P_2)$ with $|X_1|=|X_2|=|X_3|$. Let \[E:=\{(x_1,x_2,x_3)\in X_1\times X_2\times X_3: x_1,x_2,x_3\text{ distinct and collinear.}\}\]
Suppose $|E|\gg |X_1|^{2-\delta}$. Then there is a set $S\subseteq \{\gamma_{x,y}:x,y\in X_3\}$ with $|S|\gg |X_1|^{1-4\delta}$ such that \[
|\{sx=y: s\in S,x,y\in X_1\}|\gg |S||X_1|^{1-4\delta}.
\]
\end{lemma}
\begin{proof}
Let $A:=\{(x_1,x_1',x_3,x_3')\in X_1\times X_1\times X_3\times X_3: x_1',x_3',\ell_{x_1,x_3}\cap P_2\text{ is collinear}\}$. We first claim that $|A|\gg |X_1|^{3-(15/4)\delta}.$ 
Indeed, let \[X_2':=\{x\in X_2:  |\{(x_1,x_3)\in X_1\times X_3, (x_1,x,x_3)\in E\}|\gg |X_1|^{1-(5/4)\delta}\}.\] We have $|X_2'|\gg |X|^{1-(5/4)\delta}$. Suppose not, then \[\max\{|E\cap (X_1\times X_2'\times X_3))|,|E\cap (X_1\times (X_2\setminus X_2')\times X_3)|\}\ll |X_1|^{1-(5/4)\delta}|X_1|,\] hence $|E|\ll |X_1|^{2-(5/4)\delta}$, a contradiction.

Given any $x\in X_2'$, we have \[|\{(x_1,x_1',x_3,x_3')\in X_1\times X_1\times X_3\times X_3: (x_1,x,x_3)\in E , (x_1',x,x_3')\in E\}|\geq |X_1|^{2(1-(5/4)\delta)}.\] Therefore, $|A|\gg |X_1|^{3-(15/4)\delta}.$ 

Let $\tilde{B}:=\{(x_3,x_3')\in X_3^2: |\{(x_1,x_1'):(x_1,x_1',x_3,x_3')\in A\}|\gg |X_1|^{1-4\delta}\}$. Then by the same argument $|\tilde{B}|\gg |X_1|^{2-4\delta}$. Consider the map $f:\tilde{B}\to G$ by sending $(x,y)$ to $\gamma_{x,y}$. Let $N:=\frac{2}{\delta}$. For $0\leq k<N$ define \[B_k:=\{(x,y)\in \tilde{B}: |X_1|^{k\delta}\leq |f^{-1}(\gamma_{x,y})|\leq|X_1|^{(k+1)\delta}\}.\] Since $N$ is fixed when $\delta$ is fixed, there is $k_0$, such that $B_{k_0}\gg |X_1|^{2-4\delta}$ (whenever $|X_1|\gg N$).

Let $B:=B_{k_0}$ and $S:=\{\gamma_{x,y}: (x,y)\in B\}$. Then by the definition of $\title{B}$, we have \[
|\{sx=y: s\in S,x,y\in X_1\}|\gg |S||X_1|^{1-4\delta}.
\] By pigeon-hole, there is $x\in X_3$ with $|\{y\in X_3,(x,y)\in B\}|\gg |X_1|^{1-4\delta}$ and $\gamma_{x,y}\neq\gamma_{x,y'}$ if $y\neq y'$. Hence, $|S|\gg |X_1|^{1-4\delta}$ as desired. 

Moreover, we claim that for all $\eta>0$, for all $B'\subseteq B$, suppose $|B'|\leq |X_1|^{2-\eta}$, then \begin{equation}\label{equation: drop}
    \frac{|f(B')|}{|S|}\leq |X_1|^{5\delta-\eta}.
\end{equation}
Note that $|S||X_1|^{k_0\delta}\leq |B|\leq |S||X_1|^{(k_0+1)\delta}$ and $|B'|\geq |f(B')||X_1|^{k_0\delta}$ by the definition of $B=B_{k_0}$. Hence, 
\begin{align*}
|f(B')|&\leq \frac{|B'|}{|X_1|^{k_0\delta}}\leq \frac{|X_1|^{2-\eta}}{|X_1|^{k_0\delta}}\leq \frac{|X_1|^{2-4\delta}|X_1|^{4\delta-\eta}}{|X_1|^{k_0\delta}} \\
&\leq \frac{|B||X_1|^{4\delta-\eta}}{|X_1|^{k_0\delta}}\leq\frac{|S||X_1|^{(k_0+1)\delta}|X_1|^{4\delta-\eta}}{|X_1|^{k_0\delta}}=|S||X_1|^{5\delta-\eta}.
\qedhere
\end{align*}
\end{proof}

The next lemma verifies the stabilizer condition. 

\begin{lemma}\label{lem: (5)}
Let $G:=G_a(\mathsf{K})^2\rtimes G_m(\mathsf{K})$ acting on $P_1:=\{[0:\xi_1:\xi_2:\xi_3]:\xi_i\in \mathsf{K}\}$ by $(a,b,c)\star[0:\xi_1:\xi_2:\xi_3]=[0:c\xi_1:\xi_2+a\xi_1:\xi_3+b\xi_1]$.
    Let $X_1\subseteq P_1(\mathsf{K})$ be a finite set. Suppose there is $\varepsilon>0$ such that $|X_1\cap\ell|\ll |X_1|^{1-\varepsilon}$ for any projective line $\ell\in P_1(\mathsf{K})$.  Then 
    \[|\{(p_1,p_2)\in X_1^2:\mathrm{Stab}_G(p_1,p_2)\neq \mathrm{id}_G\}|\leq |X_1|^{2-\varepsilon}.\]
\end{lemma}
\begin{proof}
Let $p=[0:\xi_1:\xi_2:\xi_3]\in P_1$ and $g=(a,b,c)\in G$. Note that if $\xi_1=0$, then $\mathrm{Stab}_G(p)=G$. Suppose $\xi_1\neq 0$, $g\neq \mathrm{id}_G$ and $g\in\mathrm{Stab}_G(p)$, then either 
\[
\frac{a}{b}=\frac{\xi_2}{\xi_3} \text{ if } \xi_3\neq 0, \text{ or } b=0 \text{ if } \xi_3=0.
\]
Therefore, if $\mathrm{Stab}_G(p,p')\neq \mathrm{id}_G$ for $p=[0:\xi_1:\xi_2:\xi_3]$ and $p'=[0:\xi_1':\xi_2':\xi_3']$, then one of the following holds:
\begin{enumerate}
    \item 
    $x=0$ for $x\in\{\xi_1,\xi_2,\xi_3,\xi_1',\xi_2',
    \xi_3'\}$;
    \item 
    $\frac{\xi_2}{\xi_3} =\frac{\xi_2'}{\xi_3'}=t$ for some $t\in \mathsf{K}$.
\end{enumerate}
Fix any $p\in P_1$, the set of $p'\in P_1$ such that $\mathrm{Stab}_G(p,p')\neq \mathrm{id}_G$ is in a finite union of lines given by case (1) and (2). Thus, the conclusion follows by assumption. 
\end{proof}

To apply Theorem~\ref{thm: main1} or Theorem~\ref{thm: main}, a major assumption needs to be satisfied is that $S$ needs to escape subgroups that are sufficiently nilpotent. The next two lemmas characterize ``sufficiently nilpotent'' groups in $G_a^2(\mathsf{K})\rtimes G_m(\mathsf{K})$ for any field $\mathsf{K}$. 

\begin{lemma}\label{lem: nilpotent subgroup G_a general}
    Let $G\leq G_a(\mathsf{K})^2\rtimes G_m(\mathsf{K})$ for some field $\mathsf{K}$. Suppose $G$ is nilpotent, then $G$ is abelian. 
\end{lemma}
\begin{proof}
    Note that for any $g:=(a,b,1)$ and $g':=(a',b',c)$, we have $[g,g']=(a(c-1),b(c-1),1)\neq (0,0,1)$ if $c\neq 1$ and $(a,b)\neq (0,0)$.

If $H$ is non-abelian, then there are $g:=(a,b,1)\in [H,H]$ (since $[H,H]\leq G_a^2$) for some $(a,b)\neq (0,0)$ and $g':=(a',b',c)\in H$ with $c\neq 1$. Hence, $[g,g']$ is non-trivial in $G_a^2$, and similarly $[[g,g'],g']$ is non-trivial and so on. We conclude that $H$ is not nilpotent, contradicting our assumption.
\end{proof}

\begin{lemma}\label{lem: nilpotent subgroup}
    Let $\mathsf{K}$ be a field of characteristic $p>0$, and $N>0$ be an integer. Let $G=G_a(\mathsf{K})^2\rtimes G_m(\mathsf{K})$, and $S\subseteq G$ with $|S|< p^{1/N}$. Suppose $D\subseteq S^N$ is a subgroup of $G$. Then $D\cap G_a^2=\mathrm{id}_G$. Furthermore, suppose $H\leq G$ such that $D\trianglelefteq H$ and $H/D$ is nilpotent, then $H$ is abelian.
\end{lemma}
\begin{proof}
Suppose $D\cap G_a^2\neq 1$. Then there is $g:=(a,b,1)\in D$ with $(a,b)\neq (0,0)$. Now $\langle g\rangle\leq G_a(K)^2$ and $|\langle g\rangle|= p$. However, $|D|< p$ by assumption (as $|S|< p^{1/N}$ and $D\subseteq S^N$). 

 Since $[H,H]\leq G_a^2$, if $H/D$ is nilpotent, then $[\cdots[[H/D,H/D],H/D],\cdots,H/D]=\mathrm{id}_G$. Thus, $[\cdots[[H,H],H],\cdots,H]D/D=\mathrm{id}_G$. Hence, $[\cdots[[H,H],H],\cdots,H]\leq D\cap G_a^2=\mathrm{id}_G$. Namely, $H$ is nilpotent, and we conclude by the previous lemma.
\end{proof}

\begin{lemma}\label{lem: abelian group}
Let $g:=(a,b,m)\in G$ with $m\neq 1$. Then $C_G(g)\leq (G_m)^{g'}$, where $g':=(\frac{a}{1-m},\frac{b}{1-m},1)$.
\end{lemma}
\begin{proof}
    If $(x,y,z)\in C_G((a,b,m))$, then $x=\frac{a(z-1)}{m-1}$ and $y=\frac{b(z-1)}{m-1}$. Namely, $(x,y,z)=(0,0,z)^{g'}$.
\end{proof}

We now begin to prove Theorem~\ref{thm: three planeIntro}. In fact, we will prove a slightly stronger version.

\begin{theorem}\label{thm: three plane}
    For any $\varepsilon>0$, there is $N\in\mathbb{N}$ and $\delta>0$, such that for all field $\mathsf{K}$ (uniform in all characteristics), the following holds for any projective planes $P_1,P_2\subseteq \mathbb{P}^3(\mathsf{K})$: Suppose there are finite sets $X_i\subseteq P_i$ for $i=1,2$ and $X_3\subseteq \mathbb{P}^3(\mathsf{K})\setminus (P_1\cup P_2)$ with 
    \begin{enumerate}
        \item 
        $1\ll |X_1|=|X_2|=|X_3|$;
        \item 
        $|X_i|< p^{1/N}$ if $\operatorname{char}(\mathsf{K})=p>0$;
        \item 
        $|X_3\cap P|\leq |X_3|^{1-\varepsilon}$ for any projective plane $P$ which intersects with $P_1$ and $P_2$ in a common line;
        \item 
        $|X_i\cap \ell|\leq |X_i|^{1-\varepsilon}$ for any projective line $\ell\subseteq P_i$ for $i=1,2,3$.
    \end{enumerate}
    Then, \[|\{(x_1,x_2,x_3)\in X_1\times X_2\times X_3: x_1,x_2,x_3 \text{ distinct and collinear}\}|\leq |X_1|^{2-\delta}.\]
\end{theorem}
\begin{proof}
   Given $\varepsilon>0$ and $X_i\subseteq P_i$ for $i=1,2,3$ satisfying conditions (1)-(3) in Theorem~\ref{thm: three plane}. Suppose towards contradiction \[|\{(x_1,x_2,x_3)\in X_1\times X_2\times X_3: x_1,x_2,x_3 \text{ distinct and collinear}\}|\geq |X_1|^{2-\delta}\] for some $\delta>0$ we will choose later.
   
   Let $S\subseteq G_a(\mathsf{K})^2\rtimes G_m(\mathsf{K})=:G$ and $B\subseteq X_3$ with $S=\{\gamma_{x,y}:x,y\in B\}$ be defined as in Lemma~\ref{lem: S}, then $|X_1|^2\geq |S|\gg |X_1|^{1-4\delta}$ and \begin{equation}\label{eq: SX}
       |\{sx=y\mid s\in S,x,y\in X_1\}|\gg |S||X_1|^{1-4\delta}.
   \end{equation}

Let $\widetilde{G}:=\langle S\rangle$, we want to apply Theorem~\ref{thm: main} for $\widetilde{G}\leq G\leq \mathrm{GL}_3(\mathsf{K})$. Note that by Lemma~\ref{lem: (5)}, condition (3) of Theorem~\ref{thm: main} is satisfied with $k=2$ and $\varepsilon_3=\varepsilon$. Condition (2) is also satisfied with $r=2$ as $|S|\gg |X_1|^{1-4\delta}$ whenever $\delta<\frac{1}{8}$.

We now verify condition (2). Suppose $\operatorname{char}(\mathsf{K})=p$. Let $N_0$ be given by Theorem~\ref{thm: main} for $n=3$ ($G\leq \mathrm{GL}_3(\mathsf{K})$), $r=\frac{1}{2}$, $k=2$, $\varepsilon_1=\frac{1}{6}$, $\varepsilon_2=\frac{\varepsilon}{2}$ and $\varepsilon_3=\varepsilon$. Let $N:=2N_0$. We may assume $N>2$. Now $|S|\leq |X_1|^2<p^{2/N}=p^{1/N_0}$. Suppose there are $H\leq \tilde{G}$ and $D\subseteq S^{N_0}$ such that $H/D$ is nilpotent. Then by Lemma~\ref{lem: nilpotent subgroup}, $H$ is abelian. Suppose $\operatorname{char}(K)=0$ and $H\leq \tilde{G}$ a nilpotent subgroup. Then $H$ is also abelian by Lemma~\ref{lem: nilpotent subgroup G_a general}. Condition (3) follows from the claim below. \medskip

\noindent{\bf Claim.}
%\begin{claim}\label{lem: escape subgroup ?}
    Suppose $H\leq G$ is an abelian subgroup. Then $|S\cap gH|<|S|^{1-\frac{\varepsilon}{2}}$ for all $g\in G$.\medskip

\noindent{\emph{Proof of the Claim}.}
    Note that if $H\leq G$ is an abelian subgroup, then either $H\leq G_a^2$, or there is $(x,y,m)\in H$ with $m\neq 1$. Hence, $H\leq C_G((x,y,m))\leq (G_m)^{g'}$ where $g'=(\frac{x}{1-m},\frac{y}{1-m},1)$ by Lemma~\ref{lem: abelian group}. 

Fix $g=(a_0,b_0,m_0)\in G$. Let $B':=\{(x,y)\in B: \gamma_{x,y}\in gH\}$. We claim that $|B'|\leq |X_1|^{2-\varepsilon}$. Then $|S\cap gH|\leq |S||X_1|^{5\delta-\varepsilon}$ by (\ref{equation: drop}). Since $|S|\geq |X_1|^{1-4\delta}$, we have \[|S\cap gH|\leq |S||X_1|^{5\delta-\varepsilon}\leq |S||S|^{\frac{5\delta-\varepsilon}{1-4\delta}}=|S|^{1-\frac{\varepsilon-5\delta}{1-4\delta}}\leq |S|^{1-\frac{\varepsilon}{2}},\] where the last inequality holds when $\delta\leq\frac{\varepsilon}{10}$.

Now we prove the claim.
     Let $\gamma_{x,y}=(c'a-c,d'a-d,b'a)\in gH$ with $x=[a:1:c:d]$ and $y=[1:b':c':d']$. Suppose $H\leq G_a^2$. Then $gH\subseteq\{(t,t',m_0):t,t'\in \mathsf{K}\}$. Hence, $b'a=m_0$ and $a, b'\neq 0$ (since $x,y\not\in P_1\cup P_2$). Define $P_4:=\{[\xi_0:\frac{m_0}{a}\xi_0:\xi_2:\xi_3]:\xi_i\in \mathsf{K}\}$ then $y\in P_4$. Note that 
     $P_4$ intersects with $P_1$ and $P_2$ in the common line $\{[0:0:\xi_3:\xi_4]:\xi_i\in \mathsf{K}\}$. Hence, fix $x\in X_3$, we have $|\{y\in X_3:\gamma_{x,y}\in gH\}|\leq |X_1|^{1-\varepsilon}$ by assumption, and the claim follows in this case.
     
     For the other case, $H\leq (G_m)^{g'}$. We can write $(G_m)^{g'}$ as $\{(\lambda_1(z-1),\lambda_2(z-1),z):z\in \mathsf{K}^*\}$ for some fixed $\lambda_1,\lambda_2$ (see the proof of Lemma~\ref{lem: abelian group}). If $\gamma_{x,y}\in gH$, then \begin{equation*}
         c'a-c=(a_0+\lambda_1)z-\lambda_1;\quad
         d'a-d=(b_0+\lambda_2)z-\lambda_2;\quad
         b'a=m_0z.
     \end{equation*}
     Hence, $y=z[0:\frac{m_0}{a}:\frac{a_0+\lambda_1}{a}:\frac{b_0+\lambda_2}{a}]+[1:0:\frac{c-\lambda_1}{a}:\frac{d-\lambda_2}{a}]$. Thus, $y$ is on the line spanned by $[0:\frac{m_0}{a}:\frac{a_0+\lambda_1}{a}:\frac{b_0+\lambda_2}{a}]$ and $[1:0:\frac{c-\lambda_1}{a}:\frac{d-\lambda_2}{a}]$ for given $x=[a:1:c:d]$. By assumption, $|\{y\in X_3:\gamma_{x,y}\in gH\}|\leq |X_1|^{1-\varepsilon}$ for all $x$, and the claim  follows.
\hfill$\bowtie$
\medskip

By Theorem~\ref{thm: main}, there is $\delta'$ (only depends on $\varepsilon$, as all the constants only depend on $\varepsilon$ in our case), such that \[|\{(x_1,x_2,x_3)\in X_1\times X_2\times X_3: x_1,x_2,x_3 \text{ distinct and collinear}\}|<|X_1|^{1-\delta'}|S|.\] Choose $\delta< \frac{\delta'}{4}$, then we get a contradiction with inequality (\ref{eq: SX}). 
\end{proof}

\section{Collinear points in quadric surfaces}

Let $Q\subseteq\mathbb{P}^3(\mathsf{K})\setminus P$ be a quadric surface defined by a homogeneous equation \[\sum_{1\leq i,j\leq 4}a_{ij}x_ix_j=0\] We say $Q$ is non-degenerate or smooth if the associated symmetric matrix $B:=(a_{ij})_{i,j}$ has full rank.

Recall that the dot product $\cdot: \mathsf{K}^n\to \mathsf{K}^n$ is a bilinear form on the vector space $\mathsf{K}^n$ defined as $v\cdot w=\sum_{1\leq i\leq n} v_iw_i$ where $v=(v_1,\ldots,v_n)$ and $w=(w_1,\ldots,w_n)$.

For any point $x\in \mathbb{P}^3(\mathsf{K})$, 
%let $\tilde x$ be the line $\{kv_x:k\in K\}\subseteq K^4$ where $x=[x_1:x_2:x_3:x_4]$ and 
let $v_x$ be the vector $(x_1,\ldots,x_4)\in \mathsf{K}^4$ for some fixed choice of representation $x=[x_1:x_2:x_3:x_4]$. Three points $x,y,z\in \mathbb{P}^3(\mathsf{K})$ are collinear if and only if the linear subspace $\langle v_x,v_y,v_z\rangle$ has dimension 2. Let $W\subseteq \mathsf{K}^4$ be a linear subspace, we denote the orthogonal complement of $W$ as $W^{\bot}:=\{v\in K:v\cdot w=0, \text{ for all }w\in W\}$. 

Fix $x\in\mathbb{P}^3(\mathsf{K})\setminus Q$, let $\gamma_x:Q\to Q$ denote the bijection given by $\gamma_x(y)=z$ if and only if $x,y,z$ are collinear. Clearly $\gamma_x$ is self-inverse.  
\begin{figure}[h]
    \centering
    \includegraphics[width=0.42\linewidth]{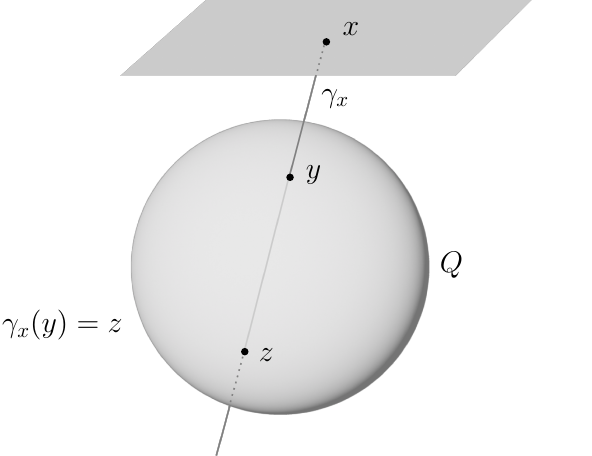}
    \caption{$\gamma_x$ maps $y$ to $z$ when $x,y,z$ are collinear.}
\end{figure}

The following lemma, originally stated in \cite[Lemma 4.7]{BDZ}, assists in determining the group structure, and we include it here for completeness.
\begin{lemma}\label{lem: PO4}
    Let $Q\subseteq \mathbb{P}^3(\mathsf{K})$ be defined by $\sum_{1\leq i\leq 4}x_i^2=0$. Let $S\subseteq \mathbb{P}^3(\mathsf{K})\setminus Q$. The group $G$ generated by $\{\gamma_x:x\in S\}$ can be naturally identified with a subgroup of $\mathrm{PO}_4(\mathsf{K})\leq \mathrm{PGL}_4(\mathsf{K})$. Here $\mathrm{PO}_4(\mathsf{K})$ is the orthogonal group preserving the dot product in $\mathsf{K}^4$.
\end{lemma}
\begin{proof}
    Fix $x\in\mathbb{P}^3(\mathsf{K})\setminus Q$, let $\tilde{\gamma}_x\in \mathrm{O} _4(\mathsf{K})$ be the linear transformation which sends $v_x$ to $-v_x$ and acts as identity on $\langle v_x\rangle^{\bot}$. As $x$ is not in $Q$, we have $\langle v_x\rangle\oplus\langle v_x\rangle^{\bot}=\mathsf{K}^4$ so that $\tilde{\gamma}_x$ is a well-defined element in $\mathrm{O}_4(\mathsf{K})$. 
    
    Let $p:\mathrm{O}_4(\mathsf{K})\to \mathrm{PO}_4(\mathsf{K})$ and $\pi:\mathsf{K}^4\setminus \{\mathrm{0}\}\to\mathbb{P}^3(\mathsf{K})$ be the natural projections. Then $\tilde{\gamma}_x$ acts on $\pi^{-1}(Q)$ since $\tilde{\gamma}_x(kv_y)\cdot\tilde{\gamma}_x(kv_y)=k^2(v_y\cdot v_y)=0$ if $y\in Q$. Moreover, write $v_y=kv_x+w$ where $w\in \langle v_x\rangle^{\bot}$, then $\tilde{\gamma}_x(v_y)=-kv_x+w=-2kv_x+v_y$. Therefore, $\tilde{\gamma}_x(v_y)\in \langle v_x, v_y\rangle$, which means $\pi(\tilde{\gamma}_x(v_y)),x,y$ are collinear. We conclude that $\gamma_x$ has the same action as $p(\tilde{\gamma}_x)$ on $Q$. And the map $\gamma_x\mapsto p(\tilde{\gamma_x})$ is injective, hence we may regard $G$ as a subgroup of $\mathrm{PO}_4(\mathsf{K})$.
\end{proof}

In the remainder of this section, we will focus on verifying the conditions required to apply Theorem~\ref{thm: main}. In particular, we aim to develop a better understanding of the relevant subgroups. We begin by proving two useful lemmas.

\begin{lemma}\label{lem: tilde g}
    Let $a\neq b\in \mathbb{P}^3(\mathsf{K})\setminus Q$ and $g\in \mathrm{PO}_4(\mathsf{K})$. Suppose $v_a\cdot v_b\neq 0$ and $g\gamma_a\gamma_b=\gamma_b\gamma_a g$. Let $\tilde g$ be a lift of $g$ in $\mathrm{O}_4(\mathsf{K})$. Then $\tilde g$ preserves the vector subspace $\langle v_a,v_b\rangle$ in $\mathsf{K}^4$. 
    
    Moreover, if $\tilde g=\tilde\gamma_x$ for some $x\in \mathbb{P}^3(\mathsf{K})\setminus Q$, then $v_x\in\langle v_a,v_b\rangle$.
\end{lemma}
\begin{proof}
    Let $\tilde{\gamma}_a$ and $\tilde{\gamma}_b$ be the lifts of $\gamma_a$ and $\gamma_b$ in $\mathrm{O}_4(\mathsf{K})$ as given in the previous lemma. Then we have that $g_{ab}:=\tilde{\gamma}_a\tilde{\gamma}_b\in \mathrm{SO}_4(\mathsf{K})$. We similarly define $g_{ba}:=\tilde{\gamma}_b\tilde{\gamma}_a$. By the assumption that $g\gamma_a\gamma_b=\gamma_b\gamma_a g$, we have $\lambda\tilde{g}g_{ab}=g_{ba} \tilde{g}$ for some $\lambda\in \mathsf{K}\setminus\{0\}$.
    
    Note that the two dimensional vector subspace $W:=\langle v_a,v_b\rangle^\bot$ is fixed by $g_{ab}$ and $g_{ba}$ from the description of $\tilde{\gamma}_a$ and $\tilde{\gamma}_b$. This implies that 
    \[
    g_{ba} \tilde{g}(v)=\lambda\tilde{g}g_{ab}(v)=\lambda\tilde g(v)\]
    for all $v\in W$. Therefore, $\tilde g(W)$ is contained in the $\lambda$-eigenspace of $g_{ba}$. 
    
    The proof now is split into two cases. Let us first consider that $\lambda\neq 1$. Then for any $v\in \tilde{g}(W)$ and $w\in W$, 
    \[
    v\cdot w=g_{ba}(v)\cdot g_{ba}(w)=\lambda v\cdot w,
    \]
    as $g_{ba}$ preserves the dot product. Hence $v\cdot w=0$ since $\lambda\neq 1$. This implies that $W^\bot=\tilde g(W)$ and $g_{ba}$ can be diagonalised as $(\lambda,\lambda,1,1)$. Note also that $\det(g_{ba})=1$ and $\lambda\neq 1$, we get $\lambda=-1$. Therefore, $\tilde{\gamma}_b\tilde{\gamma}_a(v_b)=g_{ba}(v_b)=-v_b$. Hence, $\tilde{\gamma}_a(v_b)=v_b$ and $v_b\in \langle v_a\rangle^\bot$, which contradicts our assumption. 
    
    We may now assume that $\lambda=1$ and $\tilde g(W)$ is contained in the 1-eigenspace of $g_{ba}$. We claim that $W$ is the 1-eigenspace of $g_{ba}$. Suppose $v=g_{ba}(v)=\tilde{\gamma}_b\tilde{\gamma}_a(v)$, then $\tilde{\gamma}_a(v)=\tilde{\gamma}_b(v)$ and $v-\tilde{\gamma}_a(v)=v-\tilde{\gamma}_b(v)$ is in the $(-1)$-eigenspace of both $\tilde{\gamma}_a$ and $\tilde{\gamma}_b$. Therefore, $v-\tilde{\gamma}_a(v)\in \langle v_a\rangle\cap \langle v_b\rangle=0$ and $\tilde{\gamma}_a(v)=v=\tilde{\gamma}_b(v)$ and $v\in W$. In conclusion, $\tilde g(W)=W$, and thus $\tilde g(W^\bot)=W^\bot$. Hence, $\tilde g$ preserves $\langle v_a,v_b\rangle$ as desired.

    Finally we prove the moreover part of the lemma. Suppose $\tilde g=\tilde \gamma_x$, and $v_x\not\in \langle v_a,v_b\rangle$. By the first part of this lemma, $\tilde\gamma_x$ preserves $\langle v_a,v_b\rangle$ and $\tilde\gamma_xg_{ab}=g_{ba}\tilde\gamma_x$.  Observe that $\tilde\gamma_x(v_a)\in \langle v_x,v_a\rangle$. Indeed, as in Lemma \ref{lem: PO4}, \[
    \tilde\gamma_x(v_a)=\tilde\gamma_x(kv_x+w)=-kv_x+w=v_a-2kv_x
    \]
    where $v_a=kv_x+w$ and $w\in \langle v_x\rangle^\bot$.
     Hence $\tilde\gamma_x(v_a)\in \langle v_a,v_b\rangle\cap \langle v_a,v_x\rangle$.
    Since $v_x\not\in \langle v_a,v_b\rangle$, we have $\tilde\gamma_x(v_a)\in\langle v_a\rangle$ and $v_a\in \langle v_x\rangle^\bot$, as $\langle v_a\rangle\neq \langle v_x\rangle$. Similarly we have $v_b\in \langle v_x\rangle^\bot$. Therefore, \[
    \tilde\gamma_x\tilde\gamma_a\tilde\gamma_b(v_b)
    =\tilde\gamma_x\tilde\gamma_a(-v_b)=\tilde\gamma_b\tilde\gamma_a\tilde\gamma_x(v_b)
    =\tilde\gamma_b\tilde\gamma_a(v_b).\]
    Since $\tilde\gamma_a(-v_b)\in \langle v_a,v_b\rangle\subseteq \langle v_x\rangle^\bot$, we have $-\tilde\gamma_a(v_b)=\tilde\gamma_x\tilde\gamma_a(-v_b)=\tilde\gamma_b\tilde\gamma_a(v_b).$ Therefore, $\tilde\gamma_a(v_b)$ is in the $(-1)$-eigenspace of $\tilde\gamma_b$, hence $\tilde\gamma_a(v_b)\in\langle v_b\rangle$ and $v_a\cdot v_b=0$, and this is a contradiction.   
\end{proof}

\begin{lemma}\label{lem: 3eigenspace}
    Suppose $g\in \mathrm{O}_4(\mathsf{K})$ has a $\lambda$-eigenspace in $\mathsf{K}^4$ of dimension $\geq 3$, then $\lambda=\pm 1$ and $\lambda g$ is either identity or $\tilde{\gamma}_x$ for some $x\in \mathbb{P}^3(\mathsf{K})\setminus Q$.
\end{lemma}
\begin{proof}
If the $\lambda$-eigenspace of $g$ is at least three-dimensional, we can choose an eigenvector $v$ such that $v \cdot v \neq 0$. Since $g$ preserves the dot product, we have
\[
g(v)\cdot g(v)=\lambda^2 v\cdot v=v\cdot v\neq 0.
\]
This implies $\lambda^2 = 1$, and therefore $\lambda = \pm 1$.

Now, suppose the $\lambda$-eigenspace has dimension 4. In this case, $\lambda g$ is clearly the identity. If not, then the $\lambda$-eigenspace has dimension 3. Consequently, the characteristic polynomial $P(t)$ of $g$ factors as
\[
P(t)=(t-\lambda)^3(t-\lambda')
\]
for some $\lambda' \neq \lambda \in \mathsf{K}$. Thus, $g$ can be diagonalized with eigenvalues $(\lambda', \lambda, \lambda, \lambda)$, and the $\lambda'$-eigenspace is orthogonal to the $\lambda$-eigenspace. Since $g \in \mathrm{O}_4(\mathsf{K})$, we have the determinant condition $\lambda' \lambda^3 = \pm 1$. This implies $\lambda' = \pm 1$, and $\lambda' \neq \lambda$.

Therefore, $(1/\lambda) g = \lambda g$ has a three-dimensional $1$-eigenspace and a one-dimensional $(-1)$-eigenspace. These two eigenspaces are orthogonal. Let $v_x$ be an eigenvector in the $(-1)$-eigenspace, and let $x$ be the image of $v_x$ in $\mathbb{P}^3(\mathsf{K})$. We only need to verify that $x \not\in Q$. Since $\langle v_x \rangle^\perp$ is the three-dimensional $1$-eigenspace, it follows that $v_x \not\in \langle v_x \rangle^\perp$. Therefore, $x \not\in Q$, as desired.
\end{proof}

The following celebrated Larsen–Pink Theorem~\cite{LarsenPink} serves as a key tool in the proofs that follows. 

\begin{fact}\label{fact: LarsenPink}
For every $n$ there exists a constant $C$ depending only on $n$
such that any finite subgroup $\Gamma$ of $\mathrm{GL}_n(\mathsf{K})$ over any field $\mathsf{K}$ possesses normal subgroups $\Gamma_3\lhd \Gamma_2\lhd \Gamma_1$ such that
\begin{enumerate}
    \item $[\Gamma:\Gamma_1]\leq C$.
    \item Either $\Gamma_1=\Gamma_2$, or $p:=\mathrm{char}(\mathsf{K})$ is positive and $\Gamma_1/\Gamma_2$ is a direct product of finite simple groups of Lie type in characteristic $p$.
    \item $\Gamma_2/\Gamma_3$ is abelian of order not divisible by $\mathrm{char}(\mathsf{K})$.
    \item Either  $\Gamma_3=\{\mathrm{id}_G\}$, or  $p:=\mathrm{char}(\mathsf{K})$ is positive and $\Gamma_3$ is a $p$-group.
\end{enumerate}
\end{fact}

The following lemma is used to determine the large nilpotent subgroups. 

\begin{lemma}\label{lem: nil}
     Let $H\leq \mathrm{GL}_n(\mathsf{K})$ where $\mathsf{K}$ is a finite field of characteristic $p$, and let $D\trianglelefteq H$ be such that $H/D$ nilpotent and $|D|<p$. Then there is a constant $C$ only depending on $n$ such that $H$ has a nilpotent subgroup $\Gamma_1$ with $[H:\Gamma_1]\leq C$. 
\end{lemma}
Before proving the lemma, let us first remark that 
    since $\mathrm{PGL}_n(\mathsf{K})\leq \mathrm{GL}_{n^2}(\mathsf{K})$, the above lemma also holds for subgroups of $\mathrm{PGL}_n(\mathsf{K})$.

\begin{proof}[Proof of Lemma~\ref{lem: nil}]
    Let $\Gamma_3\trianglelefteq\Gamma_2\trianglelefteq\Gamma_1$ be normal subgroups of $H$ as in Fact~\ref{fact: LarsenPink}. We claim that $\Gamma_1$ is nilpotnent. Note that we may assume $D=D\cap\Gamma_1$.

We first claim that $\Gamma_1=\Gamma_2$. Suppose not, then by the decomposition, $\Gamma_1/\Gamma_2$ is a direct product of finite simple groups of Lie type of characteristic $p$. Since $D\Gamma_2/\Gamma_2$ is a normal subgroup of $\Gamma_1/\Gamma_2$, it is also a direct product of finite simple groups unless it is trivial. Note that $|D\Gamma_2/\Gamma_2|\leq|D|<p$, and by the fact that any finite simple group of Lie type must have order at least $p$, we conclude that $D\Gamma_2/\Gamma_2$ is trivial. Therefore, $D\leq \Gamma_2$. But now $\Gamma_1/\Gamma_2$ is a quotient of $\Gamma_1/D$ which is nilpotent by the assumption. Hence, $\Gamma_1/\Gamma_2$ is nilpotent, and we get a contradiction.

Since $\Gamma_3$ is either trivial or a $p$-group and $|D|<p$, we have $\Gamma_3\cap D=\{\mathrm{id}\}$.

%$\Gamma_1$ normal in $H$, $D$ normal in $H$, $\Gamma_1\cap D=1$. We get $\Gamma_1D$ is a direct product, they commutes with each other. 

We also have that $\Gamma_2/\Gamma_3$ is abelian. This is equivalently saying that \[
[\Gamma_2/\Gamma_3,\Gamma_2/\Gamma_3]=\{\mathrm{id}\}. 
\]
Hence, $[\Gamma_2,\Gamma_2]\leq\Gamma_3$. And since $\Gamma_2/D=\Gamma_1/D$ is nilpotent, we get 
\[
[\Gamma_2/D,\cdots,[\Gamma_2/D,\Gamma_2/D]\cdots]=\{\mathrm{id}\}.
\]
Therefore, $[\Gamma_2,\cdots,[\Gamma_2,\Gamma_2]\cdots]\leq D$. Hence, 
\[
[\Gamma_2,\cdots,[\Gamma_2,\Gamma_2]\cdots]\leq D\cap \Gamma_3=\{\mathrm{id}\}.
\]
We thus conclude that $\Gamma_2=\Gamma_1$ is nilpotent.
\end{proof}

We need the following lemma to control the nilpotent subgroups in $\mathrm{PO}_4(\CC)$ with bounded step. 

\begin{lemma}\label{lem: largeAbelianchar0}
    There is a constant $C>0$, such that any nilpotent $H\leq \mathrm{PO}_4(\CC)$ of step $\leq 3$ has an abelian subgroup $\Gamma$ with $[H:\Gamma]\leq C$.
\end{lemma}
\begin{proof}
    If $H$ is finite, then the result follows from Jordan's theorem (or from Fact~\ref{fact: LarsenPink}, which can be seen as a generalization of Jordan's theorem). 
    
    Suppose now that $H$ is infinite, then the Zariski closure $\bar{H}$ is a nilpotent algebraic subgroup of $\mathrm{PO}_4(\CC)$. Consider the subgroup $\mathrm{PSO}_4(\CC)$ of index 2 in $\mathrm{PO}_4(\CC)$. Then $\bar{H}\cap \mathrm{PSO}_4(\CC)$ is an algebraic subgroup of $\mathrm{PSO}_4(\CC)\cong \mathrm{PGL}_2(\CC)\times \mathrm{PGL}_2(\CC)$. Therefore, the connected component $T$ of $\bar{H}\cap \mathrm{PSO}_4(\CC)$ is abelian. Now $T\cap H$ is an abelian subgroup of $H$ and 
    \[
    [H:T\cap H]\leq [\bar{H}:T\cap\bar{H}]\leq [\mathrm{PO}_4(\CC):\mathrm{PSO}_4(\CC)][\bar{H}\cap \mathrm{PSO}_4(\CC):T]<\infty.
    \]
    Using a standard compactness argument, we conclude that there exists a uniform constant $C$ that serves as an upper bound in the inequality above for all $H$.
 Indeed, if such uniform bound does not exist, for each $i$ we can find a nilpotent group $H_i\leq \mathrm{PSO}_4(\CC)$ of step at most $3$ whose abelian subgroups all have index at least $i$. Taking an ultraproduct $H':=\prod_{i\to\mathcal{U}}H_i$, then the argument above would allow us to find an internal abelian subgroup $T'$ of some index bounded by $i_0$, which is a contradiction to those $H_i$ with $i>i_0$. 
\end{proof}

The next lemma handles the finite characteristic case. 

\begin{lemma}\label{lem: largeAbelian}
    There is a constant $C>0$, for any $H\leq \mathrm{PO}_4(\mathsf{K})$ where $\mathsf{K}$ is a finite field of characteristic $p$ and $D\trianglelefteq H$ with $|D|<p$ and $H/D$ nilpotent, then $H$ has an abelian subgroup $\Gamma$ such that $[H:\Gamma]\leq C$.
\end{lemma}
\begin{proof}
    By Lemma~\ref{lem: nil}, $H$ has a nilpotent subgroup $T$ with $[H:T]\leq C'$ for some absolute constant $C'$. By \cite[Theorem 5.22]{Artin}, $\mathrm{PO}_4(\mathsf{K})$ has a normal subgroup $\mathrm{P\Omega}_4(\mathsf{K})$ with $[\mathrm{PO}_4(\mathsf{K}):\mathrm{P\Omega}_4(\mathsf{K})]\leq [\mathrm{O}_4(\mathsf{K}):\Omega_4(\mathsf{K})]\leq 4$ and $\mathrm{P\Omega}_4(\mathsf{K})\cong \mathrm{PSL}_2(K)\times \mathrm{PSL}_2(\mathsf{K})$. Let $T_\Omega:=T\cap \mathrm{P\Omega}_4(\mathsf{K})$. Then \[
    [H:T_\Omega]=[H:T][T:T_{\Omega}]\leq [H:T][\mathrm{PO}_4(\mathsf{K}):\mathrm{P\Omega}_4(\mathsf{K})]\leq 4C'.
    \]
    Thus $T_\Omega\leq \mathrm{PSL}_2(\mathsf{K})\times \mathrm{PSL}_2(\mathsf{K})$ is nilpotent. 

    Let $T_{\Omega,i}$ for $i = 1, 2$ denote the projection of $T_{\Omega}$ onto the two factors $\mathrm{PSL}_2(\mathsf{K})$, respectively. Then $T{\Omega,i}$ is also nilpotent. If $|T_{\Omega,i}| \leq 60$, the trivial group is the abelian subgroup of $T_{\Omega,i}$ with an index bounded by 60. Otherwise, by Dickson's classification of subgroups of $\mathrm{PSL}_2(\mathsf{K})$ (see \cite[8.27]{Huppert}), we conclude that if $|T{\Omega,i}| > 60$, it must be one of the following:
    \begin{enumerate}
        \item 
        an abelian group;
        \item 
        a diherdal group;
        \item 
        a semi-direct product of an elementary $p$ group $A$ of order $p^m$ and a cyclic group $B$ of order $t$ where $t\mid (p^m-1)$ and $t\mid (|\mathsf{K}|-1)$;
        \item 
        $\mathrm{PSL}_2(\mathsf{F})$ or $\mathrm{PGL}_2(\mathsf{F})$ where $\mathsf{F}\leq \mathsf{K}$ a subfield. 
    \end{enumerate}

Clearly, in cases (1) and (2), $T_{\Omega,i}$ has an abelian subgroup of index at most 2. For case (3), note that in this scenario, $A \rtimes B$ is isomorphic to a subgroup of $\mathsf{K}^+ \rtimes \mathsf{K}^*$ (see \cite[(6.8)]{Suzuki}). Thus, any nilpotent subgroup in this case is abelian. Additionally, $T_{\Omega,i}$ cannot fall under case (4) because $T_{\Omega,i}$ is nilpotent and has size greater than 60.

In conclusion, each $T_{\Omega,i}$ has an abelian subgroup of index at most 60, and $T_{\Omega,1} \times T_{\Omega,2}$ has an abelian subgroup of index at most $60^2$. Since $T_{\Omega} \leq (T_{\Omega,1} \times T_{\Omega,2})$, it follows that $T_{\Omega}$ also has an abelian subgroup $\Gamma$ of index at most $60^2$. Therefore, $[H : \Gamma] \leq 120^2 C'$, and we conclude by setting $C = 120^2 C'$.
\end{proof}

Using the results above, the following lemma serves as the key step that allows us to escape from nilpotent subgroups.

\begin{lemma}\label{lem: abelian intersection}
    Let $\mathsf{K}$ be any field. Fix $\varepsilon>0$. Suppose $S\subseteq\mathbb{P}^3(\mathsf{K})\setminus Q$ is a finite subset such that $|S|\gg 1$ and $|S\cap\ell|<|S|^{1-\varepsilon}$ for all projective line $\ell$. Let $\hat{S}:=\{\gamma_x:x\in S\}\subseteq \mathrm{PO}_4(\mathsf{K})$. Suppose $\Gamma$ is an abelian subgroup of $\mathrm{PO}_4(\mathsf{K})$. Then $|\hat{S}\cap g\Gamma|<|S|^{1-\frac{\varepsilon}{2}}$ for any $g\in \mathrm{PO}_4(\mathsf{K})$.
\end{lemma}

\begin{proof}
Suppose, towards a contradiction, that $|\hat S \cap g\Gamma| \geq |S|^{1-\varepsilon/2}$. Define $A := {x \in S : \gamma_x \in g\Gamma}$. Then $|A| \geq |S|^{1-\varepsilon/2}$. Moreover, for all $x, y \in A$, we have $g^{-1}\gamma_x g^{-1}\gamma_y = g^{-1}\gamma_y g^{-1}\gamma_x$, which implies $g^{-1}\gamma_y\gamma_x = \gamma_x\gamma_y g^{-1}$. Let $h := g^{-1}$. We first prove the following claim.
\medskip
    
    \noindent{\bf Claim.}
There is $a\in A$ with $B_a:=\{x\in A: v_x\cdot v_a\neq 0\}$ of size  at least $ \frac{1}{4}|S|^{1-\varepsilon/2}$.\medskip

\noindent{\emph{Proof of the Claim}.}
Choose $a_0 \in A$ arbitrarily. Suppose the set \[
A_{a_0} := \{b \in A : v_b \cdot v_{a_0} = 0\}
\]
satisfies that $|A_{a_0}|\leq \frac{1}{2} |S|^{1 - \varepsilon/2}$, then we may set $a = a_0$. Otherwise, choose $a_1 \neq a_0$. Define $A_{a_0,a_1} := \{b \in A_{a_0} : v_b \cdot v_{a_1} = 0\}$ and let $W := \langle v_{a_0} \rangle^\bot \cap \langle v_{a_1} \rangle^\bot$. Then $W$ is two-dimensional, and therefore, $\pi(W)$ is a line in $\mathbb{P}^3(\mathsf{K})$, where $\pi : \mathsf{K}^4 \to \mathbb{P}^3(\mathsf{K})$ is the natural projection. By definition, $A_{a_0,a_1} \subseteq \pi(W)$, and hence $|A_{a_0,a_1}| \leq |S|^{1-\varepsilon}$ by assumption.

Since $|A_{a_0}| \geq \frac{1}{2} |S|^{1-\varepsilon/2}$, we get
\[
|A_{a_0}\setminus A_{a_0,a_1}|\geq \frac{1}{2}|S|^{1-\frac{\varepsilon}{2}}-|S|^{1-\varepsilon}\geq \frac{1}{4}|S|^{1-\frac{\varepsilon}{2}}
\]
for any $|S|\gg 1$. Set $a := a_1$. Then $B_a := \{x \in A : v_x \cdot v_a \neq 0\} \supseteq A_{a_0} \setminus A_{a_0,a_1}$, and thus $|B_a| \geq \frac{1}{4} |S|^{1-\varepsilon/2}$, as desired.
        \hfill$\bowtie$
\medskip

  Thus, for all $x\in B_a$, we have $h\gamma_a\gamma_x=\gamma_x\gamma_a h$. Let $\tilde h$ be a lift of $h$ in $\mathrm{O}_4(\mathsf{K})$. Then $\tilde h$ preserves $\langle v_a,v_x\rangle$ for all $x\in B_a$ by Lemma \ref{lem: tilde g}. By the same argument as the claim above, we can find $b\neq a\in B_a$ such that the set $B_{a,b}:=\{x\in B_a:v_x\cdot v_b\neq 0\}$ has size at least $\frac{1}{16}|S|^{1-\varepsilon/2}$. Consequently, $\tilde{h}$ also preserves $\langle v_b, v_x \rangle$ for all $x \in B_{a,b}$.

  Now let $C_{a,b} := \{x \in B_{a,b} : v_x \not\in \langle v_a, v_b \rangle\}$. Since the complement of $C_{a,b}$ in $B_{a,b}$ is contained in the line spanned by $v_a$ and $v_b$, we have $|C_{a,b}| \geq |B_{a,b}| - |S|^{1 - \varepsilon}$. Hence, $|C_{a,b}| \geq \frac{1}{20} |S|^{1 - \varepsilon/2}$ provided that $|S| \gg 1$.

  For any $x \in C_{a,b}$, we find that $\tilde{h}(v_x) \subseteq \langle v_x, v_a \rangle \cap \langle v_x, v_b \rangle = \langle v_x \rangle$, since $v_x \not\in \langle v_a, v_b \rangle$. In other words, $v_x$ is an eigenvector of $\tilde{h}$ for all $x \in C_{a,b}$. Since $\tilde{h}$ has at most four distinct eigenspaces, by the pigeonhole principle, there exists a $\lambda$-eigenspace $W$ of $\tilde{h}$ such that 
  \[
  |W\cap \{v_x:x\in C_{a,b}\}|\geq \frac{1}{80}|S|^{1-\frac{\varepsilon}{2}}.
  \]
  
  If the dimension of $W$ is at most $2$, then by assumption $|W\cap \{v_x:x\in S\}|\leq |S|^{1-\varepsilon}$, which leads to a contradiction. Thus, we may assume that $W$ has dimension at least $3$. Applying Lemma~\ref{lem: 3eigenspace}, we conclude that $\tilde{h}$ is either the identity or $\tilde{\gamma}_{x_0}$ for some $x_0 \in \mathbb{P}^3(\mathsf{K}) \setminus Q$.

  In the latter case, we have $v_{x_0} \in \langle v_a, v_x \rangle$ for all $x \in B_a$ by Lemma~\ref{lem: tilde g}. Therefore, $v_x \in \langle v_a, v_{x_0} \rangle$ for all $x \in B_a$. However, by assumption, the set $\{x : v_x \in \langle v_a, v_{x_0} \rangle\}$ has size at most $|S|^{1-\varepsilon}$, contradicting the fact that $|B_a| \geq \frac{1}{4} |S|^{1 - \frac{\varepsilon}{2}}$.

  It remains to handle the case when $\tilde{h}$ is the identity. Take some $y_0 \in A$. Then, for all $x \in B_a$, we have
  \[
  \gamma_{y_0}\gamma_a\gamma_x=\gamma_x\gamma_a\gamma_{y_0},
  \]
as every element commutes with each other. At this point, we are in the same case as before, and we also arrive at a contradiction, as desired.
\end{proof}

Here is a direct consequence of the lemma above.
\begin{lemma}\label{lem: (3)quadric}
   Fix $\epsilon>0$. Let $\mathsf{K}$ be a finite field of characteristic $p$. Suppose $S\subseteq\mathbb{P}^3(\mathsf{K})\setminus Q$ such that $|S|\gg 1$ and $|S\cap\ell|<|S|^{1-\varepsilon}$ for all projective line $\ell$. Let $\hat{S}:=\{\gamma_x:x\in S\}\subseteq \mathrm{PO}_4(\mathsf{K})$. Suppose $H\leq \mathrm{PO}_4(\mathsf{K})$ and $D\trianglelefteq H$ with $|D|<p$ and $H/D$ nilpotent. Then $|\hat{S}\cap gH|<|S|^{1-\varepsilon/4}$ for any $g\in \mathrm{PO}_4(\mathsf{K})$.
\end{lemma}
\begin{proof}
By Lemma~\ref{lem: largeAbelian}, $H$ contains an abelian subgroup $\Gamma$ such that $[H : \Gamma] \leq C$ for some absolute constant $C$. From the previous lemma, we know that $|\hat{S} \cap g\Gamma| < |S|^{1 - \varepsilon/2}$ for all $g \in \mathrm{PO}_4(\mathsf{K})$. Consequently, we have
$
|\hat{S} \cap gH|<C|S|^{1-\frac{\varepsilon}{2}}
$
and the result follows provided that $|S| \gg 1$.
    \end{proof}

    The next lemma verifies the stabilizer condition. 

\begin{lemma}\label{lem: (5)quadric}
    Let $\mathsf{K}$ be a field which contains square roots of $-1$. Let $g\in \mathrm{PSO}_4(\mathsf{K})$ with $\mathrm{PSO}_4(\mathsf{K})$ acting on $Q$ defined by $\sum_{1\leq i\leq 4}x_i^2=0$. Suppose $g\neq \mathrm{id}$, then $\mathrm{Fix}(g):=\{x\in Q: g(x)=x\}$ either contains at most 4 points, or is a line, or is a union of two lines.
\end{lemma}
\begin{proof}
Let $i$ be a square root of $-1$. By a linear transformation, setting
\[
x_1:=x+z, \quad x_2:= ix-iz, \quad x_3:=w-y, \quad x_4:=iw+iy,
\]
we may transfer $Q$ to be defined by $xz-yw=0$.

Consider the Segre embedding of $\mathbb{P}^1(\mathsf{K}) \times \mathbb{P}^1(\mathsf{K})$ into $\mathbb{P}^3(\mathsf{K})$, with image $Q$ given by
\[
([x:y],[w:z])\mapsto [xw:xz:yw:yz]. 
\]
The action of $g$ on $Q$ induces an automorphism $t_g$ of $\mathbb{P}^1(\mathsf{K}) \times \mathbb{P}^1(\mathsf{K})$. Since $g \in \mathrm{PSO}_4(\mathsf{K})$, we know that $t_g \in \mathrm{PGL}_2(\mathsf{K}) \times \mathrm{PGL}_2(\mathsf{K})$. Write $t_g := (g_1, g_2)$.

The set of fixed points of $t_g$ is mapped to the set of fixed points of $g$, and
\[
\mathrm{Fix}(t_g)=\left\{(a,b)\in \mathbb{P}^1(\mathsf{K})\times \mathbb{P}^1(\mathsf{K}), a\in\mathrm{Fix}(g_1),b\in\mathrm{Fix}(g_2)\right\}.
\]
Since $\mathrm{Fix}(g_1)$ either contains at most 2 points or is everything, it follows that $\mathrm{Fix}(t_g)$ contains at most 4 points or is of the form $A \times \mathbb{P}^1(\mathsf{K})$ or $\mathbb{P}^1(\mathsf{K}) \times A$, where $A$ contains one or two points. The Segre embedding maps ${a} \times \mathbb{P}^1(\mathsf{K})$ or $\mathbb{P}^1(\mathsf{K}) \times {a}$ to a line in $\mathbb{P}^3(\mathsf{K})$. Thus, we obtain the desired result.
\end{proof}

Finally, we prove the main theorem of the section.

\begin{theorem}\label{thm: quadric}
    For any $\varepsilon>0$, there is $N\in\mathbb{N}$ and $\delta>0$, such that for all $\mathsf{K}=\mathbb{F}_{p^n}$ (uniform in $p$ and $n$) or $\mathsf{K}=\CC$, the following holds for any smooth quadric surface $Q\subseteq \mathbb{P}^3(\mathsf{K})$. Suppose there are finite sets $X\subseteq Q$ and $S\subseteq \mathbb{P}^3(\mathsf{K})\setminus Q$ with 
    \begin{enumerate}
        \item 
        $1\ll |X|=|S|$;
        \item 
        $|X|< p^{1/N}$ if $\operatorname{char}(\mathsf{K})=p>0$;
        \item 
        $\max\{|X\cap \ell|, |S \cap\ell|\}\leq |X|^{1-\varepsilon}$ for any projective line $\ell\subseteq  \mathbb{P}^3(K)$.
    \end{enumerate}
    Then, \[|\{(x_1,x_2,x_3)\in X^2\times S: x_1,x_2,x_3 \text{ distinct and collinear}\}|\leq |X|^{2-\delta}.\]
\end{theorem}

\begin{proof}
Suppose we are given $\varepsilon > 0$ and sets $X \subseteq Q$, $S \subseteq \mathbb{P}^3(\mathsf{K}) \setminus Q$ as in the theorem. 
We will first reduce to the case where $Q$ is defined by $\sum_{1\leq i\leq 4}x_i^2=0$.

Note that we may use elements in $A\in\mathrm{PGL}_4(\mathsf{K})$ freely to map $(Q,X,S)$ to $(AQ,AX,AS)$. Indeed, the map is bijective and linear, the triple $(AQ,AX,AS)$ satisfies all the conditions in the theorem. If $\mathsf{K}=\CC$ we can certainly make $Q$ be defined by $\sum{1 \leq i \leq 4} x_i^2 = 0$ after applying some projective linear transformation.
 If $\mathsf{K}$ is a finite field, then by \cite[Theorem 6.9]{Jacobson}, over a finite field $\mathsf{K}$ of characteristic not equal to 2, after some projective linear transformation, we may write $Q$ as $x_1^2+x_2^2+x^2_3+dx_4^2=0$, where $d\in \mathsf{K}^*$. Consider the finite field $\mathsf{K}(\sqrt{d})$. Since $X,S$ remain finite subsets of $Q(\mathsf{K}(\sqrt{d},\sqrt{-1}))$ and $\mathbb{P}^3(\mathsf{K}(\sqrt{d},\sqrt{-1}))$, respectively, and the collinearity property is not changed, we may assume $\mathsf{K}=\mathsf{K}(\sqrt{d},\sqrt{-1})$. Then, after normalises the basis to be of length 1, we may assume $Q$ is defined by $\sum_{1\leq i\leq 4}x_i^2=0$ and that $\mathsf{K}$ contains square roots of $-1$.

Now we apply Lemma~\ref{lem: PO4}, and get a subset $\Gamma_S:=\{\gamma_s:s\in S\}$ of $\mathrm{PO}_4(\mathsf{K})\setminus \mathrm{PSO}_4(\mathsf{K})$ with $|\Gamma_S|=|S|$. Take some element $h_0\in\mathrm{PO}_4(\mathsf{K})\setminus \mathrm{PSO}_4(\mathsf{K})$, then $h_0\Gamma_S\leq \mathrm{PSO}_4(\mathsf{K})$. Let $X':=X\cup h_0(X)$ and $\Gamma_S':=h_0\Gamma_S$. Observe that
\begin{align*}
    &|\{(x_1,x_2,x_3)\in X\times X\times S: x_1,x_2,x_3 \text{ distinct and collinear}\}|\\
    \leq & |\{(x,y,\gamma_s)\in X\times X\times \Gamma_S: \gamma_s(x)=y\}|\\
    = &|\{(x,y,\gamma_s)\in X\times X\times \Gamma_S: h_0\gamma_s(x)=h_0(y)\}|\\
    \leq & |\{(x,y,g)\in X'\times X'\times \Gamma_S': g(x)=y\}|.
\end{align*}
It suffices to show that there are constants $N,C$ and $\delta>0$ such that \[|\{(x,y,g)\in (X')^2\times \Gamma_S': g(x)=y|\leq |X|^{2-\delta},\] whenever $\max\{|X\cap \ell|, |S \cap\ell|\}\leq |X|^{1-\varepsilon}$ for any projective line $\ell\subseteq  \mathbb{P}^3(\mathsf{K})$ and $C\leq|X|< p^{1/N}$ in the case that $\operatorname{char}(\mathsf{K})=p>0$. 

Let $G:=\langle \Gamma_S'\rangle$.  Let $n=4, r=2, k=5,  \varepsilon_2=\varepsilon/4, \varepsilon_3=\varepsilon$. Let $N_0$ and $\delta$ be given by Theorem~\ref{thm: main}. Let $N:=N_0+1$, then $|X|\leq p^{1/(N_0+1)}$ %and $|X|\leq p^{1/6}$ 
when $\operatorname{char}(\mathsf{K})=p>0$. We claim that condition (1)-(3) of Theorem~\ref{thm: main} holds for $G,\Gamma_S',X'$ with the chosen $n,r,k,\varepsilon_i$ and $N_0$. Then by the theorem, there is $\delta>0$ such that \[|\{(x,y,g)\in (X')^2\times \Gamma_S': g(x)=y\}|\leq |X'|^{1-\delta}|S'|\leq 2|X|^{2-\delta}\] and we are done. Now we are verifying the four conditions required in Theorem~\ref{thm: main}.

% We first verify condition (1). Suppose $G$ is finite. If $\operatorname{char}(\mathsf{K})=0$, then by Jordan's theorem $G$ has an abelian subgroup $\Gamma$ with $|G:\Gamma|\leq c$ for some absolute constant $c$. This is a direct contradiction to Lemma~\ref{lem: abelian intersection}.
% Suppose $\operatorname{char}(\mathsf{K})=p>0$, we claim that $|G|\geq p/4$. Indeed, let $\pi:\mathrm{SO}_4(\mathsf{K})\to \mathrm{PSO}_4(\mathsf{K})$ be the projection map and $\tilde{G}:=\pi^{-1}(G)$. Then  $|\tilde{G}|\leq 4|G|$. Then by the Larsen-Pink Theorem (Fact~\ref{fact: LarsenPink}), there is an absolute constant $C_0$ such that if $|\tilde{G}|<p$, then $\tilde{G}$ has an abelian subgroup $H$ of index $\leq C_0$. Thus $|h_0\Gamma_S\cap g\pi(H)|\geq |\Gamma_S|/C_0$ for some $g$, namely $|\Gamma_S\cap g'\pi(H)|\geq |S|/C_0$ for $g':=h_0^{-1}g$, contradicting Lemma~\ref{lem: abelian intersection} provided $|S|$ is large enough (which only depends on $C_0$). Therefore, $|\tilde{G}|\geq p$ and $|G|\geq p/4$. Thus condition (1) is satisfied as $|X|\leq p^{1/6}<|G|^{1/5-1/10}$. 

Condition (1) can also be easily verified with $r=2$. Next we verify the escape from nilpotent subgroups condition (2). 

Suppose $\operatorname{char}(\mathsf{K})=0$ and $H\leq G$ is nilpotent, then $H$ has a abelian subgroup $H_0$ of index bounded by some absolute constant $c$ by Lemma~\ref{lem: largeAbelianchar0}. Thus, $|\Gamma_S\cap gH_0|<|S|^{1-\varepsilon/2}$ for any $g\in\mathrm{PO}_4(\mathsf{K})$ by Lemma~\ref{lem: abelian intersection}. Therefore,
\[
|\Gamma_S\cap gH|<c|S|^{1-\frac{\varepsilon}{2}}<|S|^{1-\frac{\varepsilon}{4}}
\]
when $|S|\gg_{\varepsilon,c} 1$. In particular, $|h_0\Gamma_S\cap g'H|<|S|^{1-\varepsilon/4}$ for any $g'\in\mathrm{PSO}_4(\mathsf{K})$. 

Suppose $\operatorname{char}(\mathsf{K})=p>0$, and $H\leq G\leq \mathrm{PSO}_4(\mathsf{K})$, $D\subseteq (\Gamma_S')^{N_0}$ such that $D\trianglelefteq H$ and $H/D$ nilpotent. Then $|D|\leq |S|^{N_0}\leq p/|S|$ by assumption.  By Lemma~\ref{lem: (3)quadric}, $|\Gamma_S\cap gH|<|S|^{1-\varepsilon/4}$ for any $g\in\mathrm{PO}_4(\mathsf{K})$. Thus, $|h_0\Gamma_S\cap g'H|<|S|^{1-\varepsilon/4}$ for any $g'\in\mathrm{PSO}_4(\mathsf{K})$. 

Combined the two cases together, condition (2) is verified. 

For condition (3), let $(x_i)_{1\leq i\leq 5}$ be five distinct points in $Q$ with no three of them are collinear. Then by Lemma \ref{lem: (5)quadric}, only the identity element in $\mathrm{PSO}_4(\mathsf{K})$ can fix all of them. Let \[
S_1:=\{(x_1,\ldots,x_5)\in (X')^5: x_i=x_j\text{ for some }i\neq j\}
\]
and 
\[
S_2:=\{(x_1,\ldots,x_5)\in (X')^5: x_i\text{ is in the line spanned by } x_j, x_k\text{ for some } x_i\neq x_j, i>j>k\}.
\]
Then \[\{(x_1,\ldots,x_5)\in (X')^5, \mathrm{Stab}_{G}(x_1,\ldots,x_5)\neq \mathrm{id}_G\}\subseteq S_1\cup S_2.\] Clearly, $|S_1|\leq C_1|X|^4$ for some absolute constant $C_1$. Using the assumption that $|X\cap\ell|\leq |X|^{1-\varepsilon}$ for any projective line $\ell$, we conclude that $|X'\cap \ell|=|(X\cup h_0X)\cap \ell|\leq 2|X|^{1-\varepsilon}$. Thus, $|S_2|\leq C_2|X|^{5-\varepsilon}$ for some absolute constant $C_2$. Hence condition (3) is verified.
\end{proof}

\appendix
\section{A Balog--Szemer\'edi--Gowers theorem for measures}

In this section, we will prove Theorem~\ref{thm: BSG}. Let us first state a more precise version of this theorem.
\begin{theorem}\label{thm: BSG App}
Let $G$ be a group, and $K\geq 1$ be an integer. Suppose $\nu:G\to[0,1]$ is a finitely supported probability measure and $\|\nu*\nu\|_{L^2(G)}\geq K^{-1}\|\nu\|^2_{L^2(G)}$. Let  \[
    \nu'=\nu \e_{\frac{1}{2^8K^{2}}\|\nu\|_{L^2(G)}^2<\nu<2^4K\|\nu\|_{L^2(G)}^2}. 
    \] 
Then there is an $O(K^{O(1)})$-approximate subgroup $H$ of $G$ such that
\begin{enumerate}
    \item  $|H|\ll K^{O(1)}/\|\nu\|_{L^2(G)}^2$;
    \item  $H\subseteq (\mathrm{supp}(\nu')^{-1}\mathrm{supp}(\nu'))^3$;
    \item  $\nu(xH)\gg K^{-O(1)}$ for some $x\in G$.
\end{enumerate}
\end{theorem}

The main ingredient is the usual Balog--Szemer\'edi--Gowers theorem for sets~\cite[Theorem 5.2]{Tao}. 
\begin{fact}\label{fact: BSG for sets}
    Let $A,B\subseteq G$ be finite sets with $\|\e_A*\e_B\|^2_{L^2(G)}\geq K^{-1}|A|^{3/2}|B|^{3/2}$. Then there
are sets $A'\subseteq A$ and $B'\subseteq B$ with $|A'|\geq |A|/12K$ and $|B'|\geq |B|/12K$ such that $|A'B'|\leq 2^{20} K^8 |A|^{1/2}|B|^{1/2}$. 
\end{fact}

We will also need the following result~\cite[Theorem 4.6]{Tao}, which morally states that small doubling sets are approximate groups. 
\begin{fact}\label{fact: small doubling and approx group}
    Suppose $|AB|\leq K|A|^{1/2}|B|^{1/2}$. Then there is a set $H$ satisfying
\begin{enumerate}
    \item $H$ is a $O(K^{O(1)})$-approximate group contained in $(A^{-1}A)^3$; \footnote{The fact that $H$ is contained in $(A^{-1}A)^3$ is implicit in the proof, from the proof of \cite[Proposition 4.5]{Tao}, we find $S\subseteq A^{-1}A$ has small tripling, hence by \cite[Corollary 3.11]{Tao}, $H:=(S\cup\{1\}\cup S^{-1})^3\subseteq (A^{-1}A)^3$ is an $K^{O(1)}$-approximate subgroup.}
    \item  $|H|\ll K^{O(1)}|A|^{1/2}|B|^{1/2}$;
    \item  there is a set $X$ with $|X|\ll K^{O(1)} $ such that $A\subseteq XH$ and $B\subseteq HX$. 
\end{enumerate}
\end{fact}

\begin{proof}[Proof of Theorem~\ref{thm: BSG App}]
    Let $M>1$ to be determined later, and $\delta=1/M^2$ and
\[
\nu_1=\nu\e_{\nu\geq M\|\nu\|_{L^2(G)}^2}, \quad \nu_2=\nu\e_{\nu\leq \delta\|\nu\|_{L^2(G)}^2},
\]
and
\[
\nu_{\mathrm{str}}=\nu-\nu_1-\nu_2. 
\]
Observe that $\nu_1$ is small in $L^1(G)$, that is
\[
\|\nu_1\|_{L^1(G)}\leq \Big\|\nu\frac{\nu}{M\|\nu\|_{L^2(G)}^2}\Big\|_{L^1(G)}\leq \frac{1}{M}. 
\]
Also, $\nu_2$ is small in $L^2(G)$, that is
\[
\|\nu_2\|^2_{L^2(G)}\leq \delta\|\nu\|_{L^2(G)}^2\|\nu\|_{L^1(G)}=\delta\|\nu\|_{L^2(G)}^2. 
\]
This implies that $\|\nu_2\|_{L^2(G)}\leq\|\nu\|_{L^2(G)}/M$. By Young's inequality (that is $\|f*g\|_r\leq\|f\|_p\|g\|_q$ when $1+\frac1r=\frac1p+\frac1q$), for every $\nu_*=\nu_1,\nu_2,\nu_{\mathrm{str}}$ we have
\[
\|\nu_{*}*\nu_1\|_{L^2(G)}\leq \|\nu_*\|_{L^2(G)}\|\nu_1\|_{L^1(G)}\leq \frac{\|\nu\|_{L^2(G)}}{M},
\]
and
\[
\|\nu_{*}*\nu_2\|_{L^2(G)}\leq \|\nu_*\|_{L^1(G)}\|\nu_2\|_{L^2(G)}\leq \frac{\|\nu\|_{L^2(G)}}{M}.
\]
Similarly, 
\[
\|\nu_1*\nu_{*}\|_{L^2(G)}\leq \frac{\|\nu\|_{L^2(G)}}{M} \text{ and }\|\nu_2*\nu_{*}\|_{L^2(G)}\leq \frac{\|\nu\|_{L^2(G)}}{M}.
\]
From the assumption we have
\[
\|\nu*\nu\|_{L^2(G)}\geq\frac{1}{K}\|\nu\|_{L^2(G)}.
\]
Choose $M=2^4K$. Thus by the linearity of the convolution, and Minkowski's inequality,
\begin{equation}\label{eq: info about mu_A}
\|\nu_{\mathrm{str}}*\nu_{\mathrm{str}}\|_{L^2(G)}\geq\frac{1}{2K}\|\nu\|_{L^2(G)}. 
\end{equation}
Hence by Young's inequality again, 
\[
\|\nu_{\mathrm{str}}*\nu_{\mathrm{str}}\|_{L^2(G)}\leq \|\nu_{\mathrm{str}}\|_{L^2(G)}\|\nu_{\mathrm{str}}\|_{L^1(G)}\leq \|\nu_{\mathrm{str}}\|_{L^2(G)},
\]
which gives us
\[
\|\nu_{\mathrm{str}}\|_{L^2(G)}\geq\frac{1}{2K}\|\nu\|_{L^2(G)}. 
\]

Let $A=\mathrm{supp} (\nu_{\mathrm{str}})$. As
\[
\|\nu\|_{L^2(G)}\geq \|\nu_{\mathrm{str}}\|_{L^2(G)}\geq\frac{1}{2K}\|\nu\|_{L^2(G)},
\]
we have
\[\frac{1}{4K^2}\|\nu\|^2_{L^2(G)}\leq \|\nu_{\mathrm{str}}\|_{L^2(G)}^2\leq |A|M^2 \|\nu\|_{L^2(G)}^4.\]
and 
\[\|\nu\|_{L^2(G)}^2\geq \|\nu_{\mathrm{str}}\|^2_{L^2(G)}\geq |A|(1/M^4)\|\nu\|^4_{L^2(G)}.\]
Thus,
\[
\frac{1}{2^{10}K^4}\leq |A|\|\nu\|^2_{L^2(G)}\leq 2^{16}K^4. 
\]
Also note that for every $x\in A$ we have
\[
\delta\|\nu\|^2_{L^2(G)}\leq\nu(x)\leq M\|\nu\|^2_{L^2(G)}. 
\]
This means, for every $x\in A$
\[
\frac{\nu(x)}{M \|\nu\|_{L^2(G)}^2}\leq \e_A(x)\leq \frac{\nu(x)}{\delta \|\nu\|_{L^2(G)}^2}.
\]
Hence we have the following pointwise estimate uniformly for $x\in A$
\[
\frac{1}{2^{20}K^5} \nu(x)\leq\mu_A(x)\leq 2^{18}K^6 \nu(x). 
\]
Thus by \eqref{eq: info about mu_A}, we have
\begin{align*}
\|\mu_A*\mu_A\|_{L^2(G)}&\geq \frac{1}{2^{40}K^{10}}\|\nu_{\mathrm{str}}*\nu_{\mathrm{str}}\|_{L^2(G)}\\
&\geq\frac{1}{2^{41}K^{11}}\|\nu\|_{L^2(G)}\geq \frac{1}{2^{59}K^{17}}\|\mu_A\|_{L^2(G)}. 
\end{align*}
The above inequality can be rewritten as 
\[
\|\e_A*\e_A\|^2_{L^2(G)}\geq \frac{1}{2^{118}K^{34}}|A|^3\gg K^{-O(1)}|A|^3.
\]
and by Fact~\ref{fact: BSG for sets}, there are $A', B'\subseteq A$ with 
$|A'|,|B'|\gg K^{-O(1)}|A|,$
and $|A'B'|\ll K^{O(1)}|A|$. Using Fact~\ref{fact: small doubling and approx group}, there is an $O(K^{O(1)})$-approximate group $H\subseteq (A^{-1}A)^3$ with $|H|\ll K^{O(1)}|A|$ and $A'\subseteq XH$, $B'\subseteq HX$ for some finite $|X|\ll K^{O(1)}$. 

As $A'\subseteq A$, and $A'\subseteq XH$, by pigeonhole principle, there is $x\in G$ such that
\[
|A\cap xH|\gg K^{-O(1)} |A|.
\]
Therefore
\[
\nu(xH)\geq \nu_{\mathrm{str}}(xH)\gg K^{-O(1)}\mu_A(xH)\gg K^{-O(1)}
\]
as desired. 
\end{proof}

\bibliographystyle{amsalpha}
\bibliography{ref}
\end{document}